\documentclass[11pt,reqno]{amsart}
\usepackage[all]{xy}
\usepackage{amsfonts,amsmath,amssymb,amscd}
\usepackage{color}
\usepackage[pagebackref,colorlinks, linkcolor=blue, citecolor=red]{hyperref}
\usepackage{caption}
\usepackage{hyperref}
\usepackage{mathtools}

\usepackage[top=1in, bottom=1in, left=1in, right=1in]{geometry}
\usepackage{geometry}
\usepackage{tikz}
\usetikzlibrary{positioning}
\usetikzlibrary{arrows}
\usetikzlibrary{snakes}
\usepackage{graphicx} 

\def\opn#1#2{\def#1{\operatorname{#2}}} 
\opn{\cone}{cone}
\opn{\PF}{PF}
\opn{\F}{F}
\opn{\RF}{RF}
\opn{\Ap}{Ap}

\numberwithin{equation}{section}

\newtheorem{theorem}{Theorem}[section]
\newtheorem{proposition}[theorem]{Proposition}

\newtheorem{lemma}[theorem]{Lemma}
\newtheorem{coro}[theorem]{Corollary}
\newtheorem{definition}[theorem]{Definition}

\newtheorem{remark}[theorem]{Remark}
\newtheorem{example}[theorem]{Example}
\newtheorem{conjecture}[theorem]{Conjecture}

\newtheorem{question}[theorem]{Question}

\title{Affine semigroups of maximal projective dimension-II}

\address{IIT Gandhinagar, Palaj, Gandhinagar, Gujarat-382355 India}

\author{Om Prakash Bhardwaj}
\email{om.prakash@iitgn.ac.in}

\author{Indranath Sengupta}
\email{indranathsg@iitgn.ac.in} 


\thanks{} 

\begin{document}

\begin{abstract}
If the Krull dimension of the semigroup ring is greater than one, then affine semigroups of maximal projective dimension ($\mathrm{MPD}$) are not Cohen-Macaulay, but they may be Buchsbaum. We give a necessary and sufficient condition for simplicial $\mathrm{MPD}$-semigroups to be Buchsbaum in terms of pseudo-Frobenius elements. We give certain characterizations of $\prec$-almost symmetric $\mathcal{C}$-semigroups. When the cone is full, we prove the irreducible $\mathcal{C}$-semigroups, and $\prec$-almost symmetric $\mathcal{C}$-semigroups with Betti-type three satisfy the extended Wilf's conjecture. For $e \geq 4$, we give a class of MPD-semigroups in $\mathbb{N}^2$ such that there is no upper bound on the Betti-type in terms of embedding dimension $e$. Thus, the Betti-type may not be a bounded function of the embedding dimension. We further explore the submonoids of $\mathbb{N}^d$, which satisfy the Arf property.
\end{abstract}

\maketitle

\section{Introduction}
Throughout this article, the sets of integers and non-negative integers will be denoted by $\mathbb{Z}$ and $\mathbb{N}$, respectively. By an affine semigroup $S$, we mean a finitely generated submonoid of $\mathbb{N}^d$ for some positive integer $d.$ If $S$ is minimally generated by $ a_1, \ldots, a_n $ then $\{a_1, \ldots, a_n \}$ is called the minimal system of generators of $S$. It is known that an affine semigroup $S$ has a unique minimal set of generators. The cardinality of this set is called the embedding dimension of $S$ and is denoted by $e(S)$. If $d =1$, then $S$ is a submonoid of $\mathbb{N}$. Dividing all elements of $S$ by the greatest common divisor of its nonzero elements, we obtain a numerical semigroup. Equivalently, a numerical semigroup $S$ is a submonoid of $\mathbb{N}$ such that $\mathbb{N} \setminus S$ is finite. If $S \neq \mathbb{N},$, then the largest integer not belonging to $S$ is known as the Frobenius number of $S,$ denoted by $\F(S).$ Also, the finiteness of $\mathbb{N} \setminus S$ implies that there exists at least one element $f \in \mathbb{N} \setminus S$ such that $f + (S \setminus \{0\}) \subseteq S.$ These elements are called pseudo-Frobenius elements of the numerical semigroup $S$. The set of pseudo-Frobenius elements is denoted by $\PF(S).$ In particular, $\F(S)$ is a pseudo-Frobenius number. However, for affine semigroups in $\mathbb{N}^d$, the existence of such elements is not always guaranteed. In \cite{pfelements}, the authors consider the complement of the affine semigroup in its rational polyhedral cone and study of pseudo-Frobenius elements in affine semigroups over $\mathbb{N}^d$.

\medskip
Let $k$ be a field. The semigroup ring $k[S]$ of $S$ is a $k$-subalgebra of the polynomial ring $k[t_1,\ldots,t_d].$ In other words, $k[S] = k[{\bf t}^{a_1}, \ldots, {\bf t}^{a_n}]$, where ${\bf t}^{a_i} = t_1^{a_{i1}} \cdots t_d^{a_{id}}$ for $a_i = (a_{i1},\ldots,a_{id})$ and for all $i=1,\ldots,n.$ In \cite{pfelements}, the authors prove that an affine semigroup $S$ has pseudo-Frobenius elements if and only if the length of the graded minimal free resolution of the corresponding semigroup ring $k[S]$ is maximal. Affine semigroups having pseudo-Frobenius elements are called maximal projective dimension ($\mathrm{MPD}$) semigroups. We call the cardinality of the set of pseudo-Frobenius elements the Betti-type of $S,$ and it is denoted by $\beta$-$\mathrm{t}(S).$ Note that this is not the Cohen-Macaulay type of $S$ since $\mathrm{MPD}$-semigroups, when $\mathrm{dim}(k[S]) \geq 2,$ are not Cohen-Macaulay. 

\medskip
Let $S$ be an $\mathrm{MPD}$-semigroup and let $\mathrm{cone}(S)$ denote the rational polyhedral cone of $S.$ Set $\mathcal{H}(S) = (\mathrm{cone}(S) \setminus S) \cap \mathbb{N}^d.$ For a fixed term order $\prec$ on $\mathbb{N}^d$, we define the Frobenius element as $\F(S)_{\substack{\\\prec}} = \max_{\substack{\\\prec}} \mathcal{H}(S).$ Note that in the case of $\mathrm{MPD}$-semigroups, Frobenius elements may not always exist, with respect to any term order. However, if there is a term order $\prec$ such that $\F(S)_{\substack{\\\prec}}$ exists, then in \cite{op} (also see \cite{opfpsac}), the authors extend the notions of symmetric, pseudo-symmetric, and almost symmetric numerical semigroups to the $\mathrm{MPD}$-semigroups with respect to a term order. In \cite{op}, they studied $\prec$-symmetric and $\prec$-pseudo-symmetric $\mathrm{MPD}$-semigroups. In this article, we study $\prec$-almost symmetric $\mathrm{MPD}$-semigroups. Put $\mathrm{PF}_{\substack{\\\prec}}' (S) = \mathrm{PF}(S) \setminus \{ \mathrm{F}(S)_{\substack{\\\prec}} \}.$ If $\mathrm{PF}_{\substack{\\\prec}}' (S) \neq \emptyset$ and if for any $g \in \mathrm{PF}_{\substack{\\\prec}}' (S)$, $\mathrm{F}(S)_{\substack{\\\prec}} - g \in \mathrm{PF}_{\substack{\\\prec}}' (S)$, we say that $S$ is a $\prec$-almost symmetric $\mathrm{MPD}$-semigroup. Let $S$ be an $\mathrm{MPD}$-semigroup, $T \subseteq S$ is called an ideal of $S$ if $S+T \subseteq T$. If $\mathrm{F}(S)_{\prec}$ exists, then we define $\Omega_{\prec} = \{z \in \mathrm{cone}(S) \cap \mathbb{N}^d \mid \mathrm{F}(S)_{\prec} - z \notin S\}$. Then $S \subseteq \Omega_{\prec} \subseteq \mathbb{N}^d$, and $\Omega_{\prec}$ is an ideal of $S$. In \cite{barucci-froberg}, Barucci and Fr\"{o}berg gave a characterization of almost symmetric numerical semigroups in terms of canonical semigroup ideal. We give a similar characterization for $\prec$-almost symmetric $\mathrm{MPD}$-semigroups in terms of the ideal $\Omega_{\prec}$.

\medskip
If $\mathcal{H}(S)$ is finite then the affine semigroup $S$ is called $\mathcal{C}$-semigroup. Note that $\mathcal{C}$-semigropus are $\mathrm{MPD}$-semigroups. In 1978, Wilf \cite{wilf}, proposed a conjecture related to the Diophantine Frobenius Problem that claims that the inequality 
\[ \F(S)+1 \leq e(S) \cdot |\{ s \in S \mid s < \F(S) \}| \] 
is true for every numerical semigroup. While this conjecture remains open, a potential extension of Wilf's conjecture to affine semigroups is studied in \cite{Wilfconjecture}. When the cone is full, we prove that $\prec$-almost symmetric $\mathcal{C}$-semigroups with Betti-type three satisfy the extended Wilf's conjecture.

\medskip
If the dimension of a $\mathrm{MPD}$-semigroup ring is one then it is isomorphic to the coordinate ring of a monomial curve defined by a numerical semigroup. In this case, the semigroup ring is always Cohen-Macaulay. If the dimension of a $\mathrm{MPD}$-semigroup ring is greater than one, then it is never Cohen-Macaulay. The Buchsbaum property is the next good thing to study for these semigroup rings. Let $(A,\mathbf{m})$ be a local ring. A system of elements $x_1, \ldots, x_s$ of $A$ is called a weak-regular sequence if $\mathbf{m}(x_1,\ldots,x_{i-1}) : x_i \subseteq (x_1,\ldots,x_{i-1})$, for $i = 1, \ldots,s$. Let $d = \mathrm{dim}(A)$,  system of $d$ elements $x_1,\ldots,x_d$ of A is called a system of parameters of $A$ if the ideal $(x_1,\ldots,x_d)$ is an $\mathbf{m}$-primary ideal. The local ring $A$ is called Buchsbaum if every system of parameters of $A$ is a weak-regular sequence. Suppose $A$ is a finitely generated homogeneous algebra over a field, and $\mathbf{m}$ is its maximal homogeneous ideal. In that case, we call $A$ a Buchsbaum ring if the local ring of $A$ at $\mathbf{m}$ is Buchsbaum. Based on the criteria of Trung \cite{trung}, we give a necessary and sufficient condition for simplicial affine $\mathrm{MPD}$-semigroups to be Buchsbaum in terms of pseudo-Frobenius elements.

\medskip
Now we summarize the contents of the paper. Unless and otherwise stated, $S$ denotes an $\mathrm{MPD}$-semigroup in $\mathbb{N}^d,$ minimally generated by elements $a_1,\ldots,a_n \in \mathbb{N}^d.$ Section 2 begins with a recollection of some basic definitions and results. Let $S$ be an $\mathrm{MPD}$-semigroup and the dimension of the semigroup ring $k[S]$ is greater than one then it is never Cohen-Macaulay, but it may be Buchsbaum. In section 3, we study the Buchsbaum property of $\mathrm{MPD}$-semigroup rings. In Proposition \ref{buchsbaummpd}, we give a necessary and sufficient condition for a simplicial $\mathrm{MPD}$-semigroup to be Buchsbaum. Due to the non-Cohen-Macaulayness of $\mathrm{MPD}$-semigroups, we deduce in Proposition \ref{noncm} that the associated graded ring of the ring of differential operators of a numerical semigroup ring is never Cohen-Macaulay. In section 4, we explore the $\prec$-almost symmetric $\mathrm{MPD}$-semigroups. When $\mathcal{H}(S)$ is finite, $S$ is known as $\mathcal{C}$-semigroups. In Proposition \ref{almostsymmetric} and \ref{almostsymmetriciff}, we give a characterization of $\prec$-almost symmetric $\mathcal{C}$-semigroups. A different
characterization of these semigroups is also given in Proposition \ref{cardH(S)} and \ref{cardhSiff} under suitable conditions. We define $\Omega_{\prec} = \{z \in \mathrm{cone}(S) \cap \mathbb{N}^d \mid \mathrm{F}(S)_{\prec} - z \notin S\}$. Then $S \subseteq \Omega_{\prec} \subseteq \mathbb{N}^d$, and $\Omega_{\prec}$ is an ideal of $S$. We study the $\prec$-symmetric, $\prec$-pseudo-symmetric, and $\prec$-almost symmetric $\mathcal{C}$-semigroups in terms of $\Omega_\prec$. Under suitable conditions, characterizations and equivalent conditions for $\prec$-almost symmetric $\mathcal{C}$-semigroups are given in Proposition \ref{almostsymmfirst} and \ref{almostsymmsecond}. In section 5, we study the extended Wilf's conjecture for $\prec$-almost symmetric $\mathcal{C}$-semigroups. When the cone is full, in Proposition \ref{wilfssufficient} and Corollary \ref{wilfsuffcor}, we give sufficient conditions for $\prec$-almost symmetric $\mathcal{C}$-semigroups to satisfy extended Wilf's conjecture. Using these sufficient conditions and characterizations in section 4, we prove in Theorem \ref{wilfbettithree} that $\prec$-almost symmetric $\mathcal{C}$-semigroups satisfy the extended Wilf's conjecture when the Betti-type is three. In section 6, we study the unboundedness of Betti-type of $\mathrm{MPD}$-semigroups in terms of embedding dimension with the help of a class of $\mathrm{MPD}$-semigroups in $\mathbb{N}^2$. Let $a \geq 3$ be an odd natural number and $p \in \mathbb{Z}^+$. We define
\[
S_{a,p} = \langle (a,0),(0,a^p),(a+2,2),(2,2+a^p) \rangle.
\]
$S_{a,p}$ is a simplicial affine semigroup in $\mathbb{N}^2$. In Proposition \ref{unboundedbettitype}, we prove that the Betti-type of $S_{a,p}$ is greater than or equal to $a^p-1$. Note that the embedding dimension of $S_{a,p}$ is $4$, but the Betti-type is unbounded. Using the technique of gluing, in remark \ref{remarkunbdd}, we give a class of $\mathrm{MPD}$-semigroups in $\mathbb{N}^2$ of each embedding dimension $e \geq 4$ such that there is no upper bound on the Betti-type in terms of the embedding dimension. Thus, we conclude that for $\mathrm{MPD}$-semigroups, the Betti-type may not be a bounded function of the embedding dimension. In the last section, we study the submonoids of $\mathbb{N}^d$, which satisfy the Arf property. In Proposition \ref{a+arf} and \ref{7.10}, we generalize the results of Arf numerical semigroups to Arf submonoids. In Corollary \ref{arf=pi}, we prove that Arf submonoids containing multiplicity are $\mathrm{PI}$-monoids.

\section{Preliminaries}
Let $S$ be a finitely generated submonoid of $\mathbb{N}^d$, say minimally generated by $\{a_1, \ldots, a_n\} \subseteq \mathbb{N}^d.$ Such submonoids are called affine semigroups. Consider the cone of $S$ in $\mathbb{Q}^d_{\geq 0}$,
\[  \mathrm{cone}(S) := \left\lbrace \sum_{i=1}^n \lambda_i a_i  \mid \lambda_i \in \mathbb{Q}_{\geq 0}, i = 1, \ldots, n \right\rbrace \] 
and set $\mathcal{H}(S) := (\mathrm{cone}(S) \setminus S) \cap \mathbb{N}^d.$

It is known that a finitely generated submonoid of $\mathbb{N}^d$ has a unique minimal set of generators. This set is also known as the Hilbert basis of $S$ (see \cite[Definition 7.17]{millersturmfels}).

\begin{definition} \label{pf}
	An element $f \in \mathcal{H}(S)$ is called a pseudo-Frobenius element of $S$ if
	$f + s \in S$ for all $s \in S \setminus \{0\}.$ The set of pseudo-Frobenius elements of $S$ is denoted by $\mathrm{PF}(S).$ In particular,
	\[ \mathrm{PF}(S) = \{ f \in \mathcal{H}(S) \mid f + a_j \in S, \ \forall j \in [1,n] \}.  \]  
\end{definition}

If the $\mathrm{cone}(S)$ is $\mathbb{Q}_{\geq 0}^d$ for any submonoid $S$ of $\mathbb{N}^d$, then we say $S$ has full cone. Observe that the set $\mathrm{PF}(S)$ may be empty, but if $\mathcal{H}(S)$ is finite then the existence of pseudo-Frobenius elements is guaranteed. Moreover, we have

\begin{coro}[{\cite[Corollary 2.15]{op}}]
Let $\mathcal{H}(S)$ be a non-empty finite set and $x \in \mathrm{cone}(S)$. Then $x \in \mathcal{H}(S)$ if and only if $f-x \in S$ for some $f \in \mathrm{PF(S)}$.
\end{coro}

Recall that a term order (also known as monomial ordering) on $\mathbb{N}^d$ is a total order compatible with the addition of $\mathbb{N}^d.$

\begin{definition}
	Let $\prec$ be a term order on $\mathbb{N}^d.$ Then $\mathrm{F}(S)_{\substack{\\\prec}} = \max_{\substack{\\\prec}} \mathcal{H}(S)$, if it exists, is called the Frobenius element of $S$ with respect to the term order $\prec.$
\end{definition}

Note that Frobenius elements may not exist. However, if $| \mathcal{H}(S) | < \infty$, then Frobenius elements do exist. Also, from \cite[Lemma 12]{pfelements}, we have that every Frobenius element is a pseudo-Frobenius element.

Let $k$ be a field. The semigroup ring $k[S]$ of $S$ is a $k$-subalgebra of the polynomial ring $k[t_1,\ldots,t_d].$ In other words, $k[S] = k[{\bf t}^{a_1}, \ldots, {\bf t}^{a_n}]$, where ${\bf t}^{a_i} = t_1^{a_{i1}} \cdots t_d^{a_{id}}$ for $a_i = (a_{i1},\ldots,a_{id})$ and for all $i=1,\ldots,n.$ Set $R = k[x_1,\ldots,x_n]$ and define a map $\pi : R \rightarrow k[S]$ given by $\pi(x_i) = {\bf t}^{a_i}$ for all $i=1,\ldots,n.$ Set $\deg x_i = a_i$ for all $i=1,\ldots,n.$ Observe that $R$ is a multi-graded ring and that $\pi$ is a degree preserving surjective $k$-algebra homomorphism. We denote by $I_S$ the kernel of $\pi.$ Then $I_S$ is a homogeneous ideal generated by binomials, called the defining ideal of $S$. Note that a binomial $\phi = \prod_{i=1}^n x_i^{\alpha_i} - \prod_{j=1}^n x_j^{\beta_j} \in I_S$ if and only if $\sum_{i=1}^n \alpha_i a_i = \sum_{j=1}^n \beta_j a_j.$ With respect to this grading, $\deg \phi = \sum_{i=1}^n \alpha_i a_i.$

\begin{definition}
	An affine semigroup $S$ is called {\bf maximal projective dimension ($\mathrm{MPD}$)} semigroup if $\mathrm{pdim}_R k[S] = n-1.$ Equivalently, $\mathrm{depth}_R k[S]=1.$
\end{definition}

In \cite[Theorem 6]{pfelements}, the authors proved that $S$ is an $\mathrm{MPD}$-semigroup if and only if $\mathrm{PF}(S) \neq \emptyset$. In particular, they prove that if $S$ is an $\mathrm{MPD}$-semigroup then $b \in S$ is the $S$-degree of the $(n-2)$th minimal syzygy of $k[S]$ if and only if $ b \in \left\lbrace a + \sum_{i=1}^n a_i  \mid a \in \mathrm{PF}(S)\right\rbrace$. Moreover, $\mathrm{PF}(S)$ has finite cardinality.

\begin{definition}
If $\mathcal{H}(S)$ is a non-empty finite set, then $S$ is said to be a $\mathcal{C}$-semigroup, where $\mathcal{C}$ denotes the cone of the semigroup. Note that $\mathcal{C}$-semigroups are $\mathrm{MPD}$-semigroups.
\end{definition}

\begin{example}
{ \rm Let $S$ be the submonoid of $\mathbb{N}^2$ generated by the elements
\[  
\{(0, 1), (3, 0), (4, 0), (1, 4), (5, 0), (2, 7)\}.
\]
Observe that
\[
\mathcal{H}(S) = \{(1, 0), (2, 0), (1, 1), (2, 1), (1, 2), (2, 2), (1, 3), (2, 3), (2, 4), (2, 5), (2, 6) \}.
\]
Thus, we have $\mathrm{PF}(S) = \{(1,3), (2,6)\}$. On the other hand, we obtain by Macaulay2 \cite{M2}, that the graded minimal free resolution of $k[S]$ as a module over $R=k[x_1,\ldots,x_6]$ module is given by
\[
0 \longrightarrow R(-(17,18))\oplus R(-(16,15)) \longrightarrow R^{12} \longrightarrow R^{27} \longrightarrow R^{28} \longrightarrow R^{12} \longrightarrow R \longrightarrow k[S] \longrightarrow 0.
\]
Therefore, the degrees of minimal generators of the fifth syzygy module are $(16, 15), (17, 18)$. Thus, we have
\[
\mathrm{PF}(S) = \{(16,15)-(15,12), (17,18)-(15,12)\}= \{(1,3),(2,6)\}.
\]}
\end{example}

The cardinality of $\mathrm{PF}(S)$ is equal to the last Betti number of $k[S]$ as a module over $R$. For Cohen-Macaulay modules, the last Betti number is known as the Cohen-Macaulay type of module. Note that the semigroup ring corresponding to a $\mathrm{MPD}$-semigroup is not Cohen-Macaulay unless it is the coordinate ring of a monomial curve. Motivated by this, we call the cardinality of $\mathrm{PF}(S)$, the Betti-type of $S$.

\section{Buchsbaumness of simplicial MPD-semigroups}
\medskip

\begin{definition}
Let $S$ be an affine semigroup generated by $ \{a_1,\ldots,a_n \} \subseteq \mathbb{N}^d$. Then $S$ is called simplicial if there exist $a_{i_1},\ldots, a_{i_d} \in \{ a_1,\ldots,a_n\}$ such that
\begin{enumerate}
\item  $a_{i_1},\ldots, a_{i_d}$ are linearly independent over $\mathbb{Q}$,

\item  for each $a \in S$, there exist $0 \neq \alpha \in \mathbb{N}$ such that $\alpha a = \sum_{j=1}^d \lambda_j a_{i_j}$, where $\lambda_j \in \mathbb{N}$.
\end{enumerate}

The elements $a_{i_1},\ldots, a_{i_d}$ are known as the extremal rays of $S$.
\end{definition}

Let $S$ be a simplicial affine semigroup minimally generated by  $\{ a_1, \ldots,a_d,a_{d+1},\ldots, a_n \} \subseteq \mathbb{N}^d$, and let $G(S)$ be the group generated by $S$ in $\mathbb{Z}^d$. Assume that $a_1, \ldots,a_d$ are the extremal rays of $S$. Define the set 
\begin{center}
$ \mathcal{D} = \{a \in G(S) \setminus S \mid a + 2a_i, a+2a_j \in S \quad \text{for some indices} \quad i \neq j, 1 \leq i,j \leq d \}.$
\end{center}

\begin{theorem}\cite[Theorem 5.3.2]{brunsaffinesemigroups}\label{Buchsbaumness}
Let $S$ be a simplicial affine semigroup with $\mathrm{rank}(G(S)) = d$. Let $a_1, \ldots,a_d$ are the vectors of $S$, which spans the cone of $S$. Then $k[S]$ is Buchsbaum iff 
\begin{center}
$  \{a \in G(S) \mid a + 2a_i, a+2a_j \in S \quad \text{for some indices} \quad i \neq j, 1 \leq i,j \leq d \} + \mathrm{Hilb}(S) \subseteq S$,
\end{center}
where $\mathrm{Hilb}(S)$ denotes the Hilbert basis of $S$.
\end{theorem}

A simplicial MPD-semigroup in $\mathbb{N}^d$, where $d > 1$, is never Cohen-Macaulay. Since $S$ is simplicial, we have $\mathrm{dim}(k[S]) = d$. On the other hand, since $S$ is $\mathrm{MPD}$, we have $\mathrm{depth}(k[S]) = 1$. Due to the non-Cohen-Macaulayness of $\mathrm{MPD}$-semigroups, we observe that the associated graded ring of the ring of differential operators of a numerical semigroup ring is never Cohen-Macaulay.

\begin{definition}
Let $A$ be a ring. Let $D(A)$ denotes the ring of differential operators on $A$ and $D^p(A)$ denotes the set of differential operators of degree less than or equal to $p$. Then $D(A)$ has a filtration $\{D^p(A)\}$. The associated graded ring $\mathrm{Gr}~ D(A)$ of $D(A)$ is defined as
\[
\mathrm{Gr}~ D(A) = \oplus_p \mathrm{Gr}^p D(A) = \oplus_p D^p(A)/D^{p-1}(A).
\]
\end{definition}
\begin{proposition}\label{noncm}
Let $S$ be a numerical semigroup and $k[S]$ be the semigroup ring associated with $S$. Then $\mathrm{Gr}~D(k[S])$ is never Cohen-Macaulay.
\end{proposition}
\begin{proof}
By \cite[Theorem 6]{eriksen}, we have $\mathrm{Gr}~D(k[S]) \simeq k[S'] $, where $S'$ is an affine semigroup in $\mathbb{N}^2$ with $\mathbb{N}^2 \setminus S'$ is finite. Therefore, there must exist an element $f \in \mathbb{N}^2 \setminus S'$ such that $f + s \in S'$ for all nonzero $s \in S'$. By \cite[Theorem 6]{pfelements}, $S'$ is an $\mathrm{MPD}$-semigroup. Therefore, $\mathrm{depth}_R k[S'] = 1$. On the other hand, it is clear that the Krull dimension of $k[S']$ is 2. Hence $\mathrm{Gr}~D(k[S]) \simeq k[S'] $ is not Cohen-Macaulay.
\end{proof}

We have seen that a simplicial $\mathrm{MPD}$-semigroup in $\mathbb{N}^d$, where $d > 1$, is never Cohen-Macaulay. However, we observe that it may be Buchsbaum. Following is the necessary and sufficient condition for a simplicial MPD-semigroup to be Buchsbaum.

\begin{proposition}\label{buchsbaummpd}
Let $S$ be a simplicial $\mathrm{MPD}$-semigroup in $\mathbb{N}^d$, $d > 1$. Then $k[S]$ is Buchsbaum if and only if $\mathcal{D} = \mathrm{PF}(S)$.
\end{proposition}
\begin{proof}
Let $a \in \mathcal{D}$ and $k[S]$ is Buchsbaum, then by Theorem \ref{Buchsbaumness}, we get $a + \text{Hilb}(S) \subseteq S$. Since $S$ is a finitely generated submonoid of $\mathbb{N}^d$, we have $\mathrm{Hilb}(S) = \{a_1,\ldots,a_n\}$. Therefore, we have $a + \{a_1, \ldots,a_n\} \in S$. Since $a \in G(S)\setminus S$, from the proof of the converse part of  \cite[Theorem 6]{pfelements}, we deduce that $a \in \mathcal{H}(S)$. Therefore, $a \in \mathrm{PF}(S)$ and hence, $\mathcal{D} \subseteq \mathrm{PF}(S)$. Conversely, suppose $a \in \mathrm{PF}(S)$, then clearly $a \in G(S)\setminus S$. Also, since $a + s \in S$ for all nonzero $s \in S$, we obtain that $a \in \mathcal{D}$. Hence, $ \mathrm{PF}(S)\subseteq \mathcal{D}$.
Now, suppose $\mathcal{D} = \mathrm{PF}(S)$. Since Hilbert basis is a subset of the semigroup, we get $\mathcal{D} + \text{Hilb}(S) \subseteq S$. Hence by Theorem \ref{Buchsbaumness}, $k[S]$ is Buchsbaum.
\end{proof}

\begin{example}
{\rm Let $S$ be a submonoid of $\mathbb{N}^2$ generated by the set
 \[S = \langle (1,0), (1,1), (1,2), (0,3), (0,4), (0,5) \rangle.\]
Note that $\mathcal{H}(S) = \{(0,1),(0,2)\}$, and also $\mathrm{PF}(S) = \{(0,1), (0,2)\}.$ Observe that, here $G(S) =\mathbb{Z}^2$. $S$ is a simplicial affine semigroup with respect to the extremal rays $\{(1,0), (0,3)\}$. Thus the set $\{(1,0), (0,3)\}$, generate the cone of $S$. Suppose $a \in \mathbb{Z}^2$ and $a+2(1,0), a+2(0,3) \in S$ simultaneously, then none of the component of $a$ can be negative. Thus one can observe that, either $a \in S$ or $a \in \{(0,1), (0,2)\}$. Therefore, we get $\mathcal{D} = \{(0,1),(0,2)\} = \mathrm{PF}(S)$. Hence, $k[S]$ is Buchsbaum.}
\end{example}

\section{$\prec$-almost symmetric semigroups}

In \cite{barucci-froberg}, Barucci and Fr\"{o}berg gave a characterization of almost symmetric numerical semigroups in terms of canonical semigroup ideal. In this section, we give a similar characterization for $\prec$-almost symmetric $\mathrm{MPD}$-semigroups.

\medskip
Let $S$ be an $\mathrm{MPD}$-semigroup, $T \subseteq S$ is called an ideal of $S$ if $S+T \subseteq T$. Let $\prec$ be a term order on $\mathbb{N}^d$. If $\mathrm{F}(S)_{\prec}$ exists, then we define $\Omega_{\prec} = \{z \in \mathrm{cone}(S) \cap \mathbb{N}^d \mid \mathrm{F}(S)_{\prec} - z \notin S\}$. It is clear that $S \subseteq \Omega_{\prec} \subseteq \mathbb{N}^d$.

\begin{proposition}
Let $S$ be a $\mathcal{C}$-semigroup. Then $\Omega_{\prec}$ is an ideal of $S$ for any term order $\prec$.
\end{proposition}
\begin{proof}
Let $s \in S$ and $z \in \Omega_{\prec}$ then we have $\mathrm{F}(S)_{\prec} - (s +z) = (\mathrm{F}(S)_{\prec}-z) - s$. Since $\mathrm{F}(S)_{\prec}-z \notin S$, we have $\mathrm{F}(S)_{\prec} - (s +z) \notin S$, and therefore $s+z \in \Omega_{\prec}$. Hence, $\Omega_{\prec}$ is an ideal of $S$.
\end{proof}

Let $I$ and $J$ be two ideal of $S$, then we define 
\begin{center}
$I-J = \{z \in \mathrm{cone}(S) \cap \mathbb{N}^d \mid z + J \subseteq I\}$. 
\end{center}

\begin{proposition}
 Let $I$ and $J$ be two ideals of $S$, then $I-J$ is also an ideal of $S$.
\end{proposition}
\begin{proof}
Let $z \in I-J$. This implies that $z+J \subseteq I$. For any $s \in S$, we have $s+z \in \mathrm{cone(S) \cap \mathbb{N}^d}$. Therefore $z+ s + J = s+(z+J) \subseteq s+I$. Since $I$ is an ideal of $S$, we get $z+s+J \subseteq I$. Therefore, $(I-J)+S \subseteq I-J$. Hence, $I-J$ is an ideal of $S$.
\end{proof}

\begin{proposition}\label{PFS}
Let $S^* = S \setminus \{0\}$. Then $\mathrm{PF}(S) =  (S-S^*)\setminus S$.
\end{proposition}
\begin{proof}
We have 
\begin{align*}
S-S^* &= \{z \in cone(S) \cap \mathbb{N}^d \mid z + S^* \subseteq S\}\\
&=\{z \in cone(S) \cap \mathbb{N}^d \mid z + s\in S, ~ \forall ~ s \in S\setminus \{0\} \}\\
&= S \cup \mathrm{PF}(S).
\end{align*}
Hence, we have $\mathrm{PF}(S) = (S-S^*)\setminus S$.
\end{proof}

Now, we recall some definitions from \cite{op}.

\begin{definition}
Let $\prec$ be a term order $\prec$ such that $\mathrm{F}(S)_{\substack{\\\prec}} = \max_{\substack{\\\prec}} \mathcal{H}(S)$ exists.
	\begin{enumerate}
		\item If $|\mathrm{PF}(S)|=1$ and $\mathrm{PF}(S) = \{\mathrm{F}(S)_{\substack{\\\prec}}\}$, then $S$ is called a {\bf $\prec$-symmetric} semigroup.
		
		\item Put $\mathrm{PF}_{\substack{\\\prec}}' (S) = \mathrm{PF}(S) \setminus \{ \mathrm{F}(S)_{\substack{\\\prec}} \}.$ If $\mathrm{PF}_{\substack{\\\prec}}' (S) \neq \emptyset$ and if for any $g \in \mathrm{PF}_{\substack{\\\prec}}' (S)$, $\mathrm{F}(S)_{\substack{\\\prec}} - g \in \mathrm{PF}_{\substack{\\\prec}}' (S)$, we say that $S$ is {\bf $\prec$-almost symmetric}. Further, if the Betti-type of $S$ is two, then $S$ is called {\bf $\prec$-pseudo-symmetric}. In this case, $\mathrm{PF}(S) = \{ \mathrm{F}(S)_{\substack{\\\prec}}, \mathrm{F}(S)_{\substack{\\\prec}}/2 \}.$ 
	\end{enumerate}
\end{definition}

\begin{coro}
Let $S$ be a $\mathcal{C}$-semigroup, then $S$ is $\prec$-symmetric if and only if $\Omega_{\prec} = S$.
\end{coro}
\begin{proof}
Follows from \cite[Theorem 3.6]{op}.
\end{proof}

\begin{coro}
Let $S$ be a $\mathcal{C}$-semigroup, then $S$ is $\prec$-pseudo-symmetric if and only if $\Omega_{\prec} = S \cup \{\mathrm{F}(S)_{\prec}/2\}$.
\end{coro}
\begin{proof}
Follows from \cite[Theorem 3.7]{op}.
\end{proof}

\begin{proposition}\label{almostsymmetric}
Let $S$ be a $\mathcal{C}$-semigroup. If $S$ is $\prec$-almost symmetric then for each $g \in \cone(S) \cap \mathbb{N}^d$ we have: 
	\[ g \in S \iff \mathrm{F}(S)_{\substack{\\\prec}} - g \notin S \text{ and } g \notin \mathrm{PF}_{\substack{\\\prec}}' (S). \]
\end{proposition}
\begin{proof}
Let $S$ be $\prec$-almost symmetric. Then for all $ f \in \mathrm{PF}_{\substack{\\\prec}}' (S)$, $\mathrm{F}(S)_{\substack{\\\prec}} - f \in \mathrm{PF}_{\substack{\\\prec}}' (S)$. Let $g \in S$, then it is clear that $g \notin \mathrm{PF}_{\substack{\\\prec}}' (S)$. If $\mathrm{F}(S)_{\substack{\\\prec}} - g \in S$ then $\mathrm{F}(S)_{\substack{\\\prec}} \in S$, which is a contradiction. Hence $\mathrm{F}(S)_{\substack{\\\prec}} - g \notin S$. For the converse, we will prove that if $g \in \cone(S) \cap \mathbb{N}^d$ and $g \notin S$ then $\mathrm{F}(S)_{\substack{\\\prec}} - g \in S \text{ or } g \in \mathrm{PF}_{\substack{\\\prec}}' (S)$. If $g \in  \mathrm{PF}_{\substack{\\\prec}}' (S)$ then nothing to prove. If $g \notin  \mathrm{PF}_{\substack{\\\prec}}' (S)$ then $g \in \mathcal{H}(S)$ and by \cite[Corollary 2.15]{op}, there exist $f \in \mathrm{PF}(S)$ such that $f - g \in S$. If $f = \mathrm{F}(S)_{\substack{\\\prec}}$ then we are done. If not then $f \in \mathrm{PF}_{\substack{\\\prec}}' (S)$. Therefore, we have $\mathrm{F}(S)_{\substack{\\\prec}} - f \in \mathrm{PF}_{\substack{\\\prec}}' (S)$. Since $f -g \in S$, we have $(\mathrm{F}(S)_{\substack{\\\prec}} - f) + (f - g) \in S$. Hence $\mathrm{F}(S)_{\substack{\\\prec}} - g \in S$.
\end{proof}

\begin{definition}
Let $\preceq_{\mathbb{N}^d}$ be the partial order on ${\mathbb{N}^d}$ such that $a = (a_1 ,\ldots, a_d ) \preceq_{\mathbb{N}^d} b = (b_1 ,\ldots,b_d )$ if
and only if $a_i \leq b_i , i \in \{1,\ldots,d\}$. This partial order is known as the usual partial order on $\mathbb{N}^d$.
\end{definition}

\begin{proposition}\label{cardH(S)}
Let $S$ be a $\mathcal{C}$-semigroup with full cone. If $S$ is $\prec$-almost symmetric then
\[
 \vert \mathcal{H}(S) \setminus  \mathrm{PF}_{\substack{\\\prec}}'(S)  \vert = \vert \{ g \in S \mid g \preceq_{\mathbb{N}^d} \mathrm{F}(S)_{\substack{\\\prec}}\} \vert
\]
\end{proposition}
\begin{proof}
Let $S$ be a $\prec$-almost symmetric semigroup. We define a map
\begin{align*}
    \phi: \mathcal{H}(S) \setminus  \mathrm{PF}_{\substack{\\\prec}}'(S)  \longrightarrow \{ g \in S \mid g \preceq_{\mathbb{N}^d} \F(S)_{\substack{\\\prec}}\} \text{ by } \phi(g) = \F(S)_{\substack{\\\prec}} - g.
    \end{align*}
Let $\mathcal{H}(S) \setminus \{ \mathrm{PF}_{\substack{\\\prec}}'(S) .$ Then $g \notin S$ and $g \notin  \mathrm{PF}_{\substack{\\\prec}}'(S).$ Since $S$ is $\prec$-almost symmetric, by Proposition \ref{almostsymmetric}, we get $\mathrm{F}(S)_{\substack{\\\prec}} - g \in S.$ As $g \in \mathbb{N}^d$ and $\F(S)_{\substack{\\\prec}} - g \in \mathbb{N}^d$, we have $\F(S)_{\substack{\\\prec}} - g \preceq_{\mathbb{N}^d} \mathrm{F}(S)_{\substack{\\\prec}}$. Thus $\phi$ is a well-defined map and clearly injective also. Now let $g \in S$ such that $g \preceq_{\mathbb{N}^d} \F(S)_{\substack{\\\prec}}$. Therefore, we have $\F(S)_{\substack{\\\prec}} - g \notin \mathrm{PF}_{\substack{\\\prec}}'(S)$. If $\F(S)_{\substack{\\\prec}} - g \in \mathrm{PF}_{\substack{\\\prec}}'(S)$ then $ \F(S)_{\substack{\\\prec}} - (\F(S)_{\substack{\\\prec}} - g) \in \mathrm{PF}_{\substack{\\\prec}}'(S)$. This implies $g \in \mathrm{PF}_{\substack{\\\prec}}'(S)$, which is a contradiction. As $S$ is $\prec$-almost symmetric, $\F(S)_{\substack{\\\prec}} - g \notin S$. Since $g \preceq_{\mathbb{N}^d} \F(S)_{\substack{\\\prec}}$, we get $\F(S)_{\substack{\\\prec}} - g \in \mathbb{N}^d = \mathrm{cone}(S)$. Thus $\F(S)_{\substack{\\\prec}} - g \in \mathcal{H}(S) \setminus  \mathrm{PF}_{\substack{\\\prec}}'(S)$ and  $\phi(\F(S)_{\substack{\\\prec}} - g) = g$. Therefore $\phi$ is a bijective map. Hence $\vert \mathcal{H}(S) \setminus  \mathrm{PF}_{\substack{\\\prec}}'(S)  \vert = \vert \{ g \in S \mid g \preceq_{\mathbb{N}^d} \mathrm{F}(S)_{\substack{\\\prec}}\} \vert$.

\end{proof}

\begin{proposition}\label{almostsymmfirst}
Let $S$ be a $\mathcal{C}$-semigroup, if $S$ is $\prec$-almost symmetric then
\begin{enumerate}
\item[(a)] $\Omega_{\prec} = S \cup \mathrm{PF}_{\substack{\\\prec}}' (S)$.
\item[(b)] $\Omega_{\prec} + S^* = S^* $.
\item[(c)] $\Omega_{\prec} \subseteq S - S^* $.
\item[(d)] $\vert \mathrm{PF}(S) \vert = \vert \Omega_{\prec} \setminus S \vert + 1$.
\end{enumerate}
\end{proposition}
\begin{proof}
(a) Let $z \in \Omega_{\prec}$, then $\mathrm{F}(S)_{\prec} - z \notin S$. By Proposition \ref{almostsymmetric}, either $z \in S$ or $z \in \mathrm{PF}_{\substack{\\\prec}}' (S)$. This implies that $\Omega_{\prec} \subseteq S \cup \mathrm{PF}_{\substack{\\\prec}}' (S)$. Now suppose $z \in \mathrm{PF}_{\substack{\\\prec}}' (S)$. Since  $S$ is $\prec$-almost symmetric, we have $\mathrm{F}(S)_{\prec} - z \in \mathrm{PF}_{\substack{\\\prec}}' (S)$. Thus $\mathrm{F}(S)_{\prec} - z \notin S$. Since $S \subseteq \Omega_{\prec}$, this proves that $ S \cup \mathrm{PF}_{\substack{\\\prec}}' (S) \subseteq \Omega_{\prec}$.

\medskip

(b) Follows immediately from part (a).

\medskip

(c) Note from Proposition \ref{PFS} that $S - S^* \subseteq S \cup \mathrm{PF}(S)$. Now it follows immediately from part (a).

\medskip

(d) Note from part (a) that $\vert \Omega_{\prec} \setminus S \vert = \vert \mathrm{PF}_{\substack{\\\prec}}' (S) \vert$, hence $\vert \mathrm{PF}(S) \vert = \vert \Omega_{\prec} \setminus S \vert + 1$.
\end{proof}

\begin{example}
{\rm Let $S$ be a submonoid of $\mathbb{N}^2$ generated by the set
\[
\{ \langle (0, 1), (3, 0), (4, 0), (1, 5), (5, 0), (2, 9) \rangle\}.
\]
Therefore,
\[\mathcal{H}(S) = \left\lbrace 
	\begin{array}{c}
	(1,0),(1,1),(1,2),(1,3),(1,4),(2,0),(2,1)\\[1mm] (2,2),(2,3), (2,4),(2,5),(2,6), (2,7), (2,8)
	\end{array}
	\right\rbrace .\]
Note that by \cite[Corollary 2.15]{op}, $\mathrm{PF}(S) = \{(1,4),(2,8)\}$. Let $\prec$ be the graded lexicographic order then $S$ is $\prec$-almost symmetric with Betti-type 2 and $\F(S)_{\substack{\\\prec}} = (2,8)$. 
Observe that $\F(S)_{\substack{\\\prec}} - (1,4) = (2,8) - (1,4) = (1,4) \notin S$, and for all other $f \in \mathcal{H}(S)$, $\F(S)_{\substack{\\\prec}} - f \in S$. Hence, we get 
\[
\Omega_{\prec} =  S \cup \{(1,4)\} = S \cup \mathrm{PF}_{\substack{\\\prec}}' (S).
\]
Also, it is clear that $\Omega_{\prec} \setminus S = \{(1,4)\}$. Thus $\vert \Omega_{\prec} \setminus S \vert + 1 = 2$, which is equal to the cardinality of $\mathrm{PF}(S)$.}
\end{example}

If we fix the cone as full cone and $f \preceq_{\mathbb{N}^d} \mathrm{F}(S)_{\substack{\\\prec}}$ for all $f \in \mathrm{PF}_{\substack{\\\prec}}'(S)$ then the converse of the above propositions is also true.

\begin{proposition}\label{almostsymmetriciff}
Let $S$ be a $\mathcal{C}$-semigroup with full cone such that $\mathrm{PF}_{\substack{\\\prec}}'(S) \neq \emptyset$ and $f \preceq_{\mathbb{N}^d} \mathrm{F}(S)_{\substack{\\\prec}}$ for all $f \in \mathrm{PF}_{\substack{\\\prec}}'(S)$. If for each $g \in \cone(S) \cap \mathbb{N}^d$ we have: 
	\[ g \in S \iff \mathrm{F}(S)_{\substack{\\\prec}} - g \notin S \text{ and } g \notin \mathrm{PF}_{\substack{\\\prec}}' (S), \]
	then $S$ is a $\prec$-almost symmetric semigroup.
\end{proposition}
\begin{proof}
Suppose that for each $g \in \mathrm{cone}(S) \cap \mathbb{N}^d$, $g \in S \iff \F(S)_{\substack{\\\prec}} - g \notin S \text{ and } g \notin \mathrm{PF}_{\substack{\\\prec}}' (S).$ Let $f \in \mathrm{PF}_{\substack{\\\prec}}' (S)$. Since $f \preceq_{\mathbb{N}^d} \mathrm{F}(S)_{\substack{\\\prec}}$, we have $\mathrm{F}(S)_{\substack{\\\prec}} - f \in \mathbb{N}^d = \mathrm{cone}(S) \cap \mathbb{N}^d $. Also since $f \in \mathrm{PF}_{\substack{\\\prec}}' (S)$, this implies $\mathrm{F}(S)_{\substack{\\\prec}} - f \notin S$. Therefore, we get $\mathrm{F}(S)_{\substack{\\\prec}} - (\mathrm{F}(S)_{\substack{\\\prec}} - f) \in S$ or $\mathrm{F}(S)_{\substack{\\\prec}} - f \in \mathrm{PF}_{\substack{\\\prec}}' (S)$. If $\mathrm{F}(S)_{\substack{\\\prec}} - (\mathrm{F}(S)_{\substack{\\\prec}} - f) \in S$ then $f \in S$, which is a contradiction. Therefore, $\mathrm{F}(S)_{\substack{\\\prec}} - f \in \mathrm{PF}_{\substack{\\\prec}}' (S)$. Hence, $S$ is a $\prec$-almost symmetric semigroup.
\end{proof}

\begin{proposition}\label{cardhSiff}
Let $S$ be a $\mathcal{C}$-semigroup with full cone such that $\mathrm{PF}_{\substack{\\\prec}}'(S) \neq \emptyset$ and $f \preceq_{\mathbb{N}^d} \mathrm{F}(S)_{\substack{\\\prec}}$ for all $f \in \mathrm{PF}_{\substack{\\\prec}}'(S)$. If 
\[
 \vert \mathcal{H}(S) \setminus  \mathrm{PF}_{\substack{\\\prec}}'(S)  \vert = \vert \{ g \in S \mid g \preceq_{\mathbb{N}^d} \mathrm{F}(S)_{\substack{\\\prec}}\} \vert,
\]
then $S$ is a $\prec$-almost symmetric semigroup.
\end{proposition}
\begin{proof}
let $\vert \mathcal{H}(S) \setminus  \mathrm{PF}_{\substack{\\\prec}}'(S)  \vert = \vert \{ g \in S \mid g \preceq_{\mathbb{N}^d} \mathrm{F}(S)_{\substack{\\\prec}}\} \vert$. Consider a map 
	\begin{align*}
		\psi : \{ g \in S \mid g \preceq_{\mathbb{N}^d} \F(S)_{\substack{\\\prec}}\} \longrightarrow  \vert \mathcal{H}(S) \setminus  \mathrm{PF}_{\substack{\\\prec}}'(S)  \vert \text{ such that } \psi(g) = \F(S)_{\substack{\\\prec}} - g.
	\end{align*}
	Let $g \in S$ and $g \preceq_{\mathbb{N}^d} \F(S)_{\substack{\\\prec}}.$ This implies $\F(S)_{\substack{\\\prec}} - g \notin S$ and $\F(S)_{\substack{\\\prec}} - g \in \mathbb{N}^d = \mathrm{cone}(S) \cap S.$ Thus $\F(S)_{\substack{\\\prec}} - g \in \mathcal{H}(S).$ Since $g \in S,$ we get $\F(S)_{\substack{\\\prec}} - g \notin \mathrm{PF}_{\substack{\\\prec}}'(S).$ Thus $\psi$ is well defined and clearly injective also. Since $\vert \mathcal{H}(S) \setminus  \mathrm{PF}_{\substack{\\\prec}}'(S)  \vert = \vert \{ g \in S \mid g \preceq_{\mathbb{N}^d} \mathrm{F}(S)_{\substack{\\\prec}}\} \vert$, $\psi$ becomes a bijective map. Therefore, there exists a map $\delta : \mathcal{H}(S) \setminus  \mathrm{PF}_{\substack{\\\prec}}'(S)  \longrightarrow \{ g \in S \mid g \preceq_{\mathbb{N}^d} \F(S)_{\substack{\\\prec}}\}$ such that  $\delta \circ \psi = \mathrm{Id}_{\{ g \in S \mid g \preceq_{\mathbb{N}^d} \F(S)_{\substack{\\\prec}}\}}$ and $\psi \circ \delta = \mathrm{Id}_{\mathcal{H}(S) \setminus  \mathrm{PF}_{\substack{\\\prec}}'(S)}.$ Let $g \in \mathcal{H}(S) \setminus  \mathrm{PF}_{\substack{\\\prec}}'(S).$ Then $g = (\psi \circ \delta) (g) = \F(S)_{\substack{\\\prec}} - \delta (g).$ This implies that $\delta(g) = \F(S)_{\substack{\\\prec}} - g.$ Now, to prove $S$ is a $\prec$-almost symmetric semigroup, we use Proposition \ref{almostsymmetriciff}. It is enough to prove that for $g \in \mathbb{N}^d \setminus \mathrm{PF}_{\substack{\\\prec}}'(S)$, if $\F(S)_{\substack{\\\prec}} - g \notin S$ then $g \in S$. Suppose $g \notin S$, then $g \in \mathcal{H}(S) \setminus  \mathrm{PF}_{\substack{\\\prec}}'(S)$. This implies that $\delta(g) = \F(S)_{\substack{\\\prec}} - g \in S$. Hence, $S$ is a $\prec$-almost symmetric semigroup.
\end{proof}

\begin{proposition}\label{almostsymmsecond}
Let $S$ be a $\mathcal{C}$-semigroup with full cone such that $\mathrm{PF}_{\substack{\\\prec}}'(S) \neq \emptyset$ and $f \preceq_{\mathbb{N}^d} \mathrm{F}(S)_{\substack{\\\prec}}$ for all $f \in \mathrm{PF}_{\substack{\\\prec}}'(S)$. Then the following are equivalent:
\begin{enumerate}
\item[(a)] $S$ is $\prec$-almost symmetric.
\item[(b)] $\Omega_{\prec} = S \cup \mathrm{PF}_{\substack{\\\prec}}' (S)$.
\item[(c)] $\Omega_{\prec} + S^* = S^* $.
\item[(d)] $\vert \mathrm{PF}(S) \vert = \vert \Omega_{\prec} \setminus S \vert + 1$.
\end{enumerate}
Moreover, (b) $\iff$ (c) $\iff$ (d) holds without the full cone condition and $f \preceq_{\mathbb{N}^d} \mathrm{F}(S)_{\substack{\\\prec}}$ for all $f \in \mathrm{PF}_{\substack{\\\prec}}'(S)$.
\end{proposition}
\begin{proof}
(a) $\iff$ (b): (a) implies (b) is immediate by Proposition \ref{almostsymmfirst}. Now suppose $\Omega_{\prec} = S \cup \mathrm{PF}_{\substack{\\\prec}}' (S)$. Let $g \in \mathrm{cone}(S) \cap \mathbb{N}^d$. If $\F(S)_{\substack{\\\prec}} - g \notin S$ then $ g \in \Omega_{\prec}$. This implies either $g \in S$ or $g \in \mathrm{PF}_{\substack{\\\prec}}' (S)$. Hence $S$ is $\prec$-almost symmetric by Proposition \ref{almostsymmetriciff}.

\medskip

(b) $\iff$ (c): (b) implies (c) is obvious. Now suppose $\Omega_{\prec} + S^* = S^* $. This implies $\Omega_{\prec}\subseteq S \cup \mathrm{PF}(S)$. Since $\F(S)_{\substack{\\\prec}} \notin \Omega_{\prec}$, we get   $\Omega_{\prec}\subseteq S \cup \mathrm{PF}_{\substack{\\\prec}}' (S)$. On the other hand, we always have $S \cup \mathrm{PF}_{\substack{\\\prec}}' (S) \subseteq \Omega_{\prec}$. Hence $\Omega_{\prec} = S \cup \mathrm{PF}_{\substack{\\\prec}}' (S)$. 

\medskip

(c) $\iff$ (d): Suppose $\Omega_{\prec} + S^* = S^* $. This implies $\Omega_{\prec} \subseteq S \cup \mathrm{PF}_{\substack{\\\prec}}' (S)$. Therefore $\vert \Omega_{\prec} \setminus S \vert \leq \vert \mathrm{PF}_{\substack{\\\prec}}' (S) \vert$. Thus $\vert \Omega_{\prec} \setminus S \vert + 1 \leq \vert \mathrm{PF}(S) \vert$. On the other hand we have $S \cup \mathrm{PF}_{\substack{\\\prec}}' (S) \subseteq \Omega_{\prec}$. This implies have $ \mathrm{PF}_{\substack{\\\prec}}' (S) \subseteq \Omega_{\prec} \setminus S$. Therefore $ \vert \mathrm{PF}_{\substack{\\\prec}}' (S) \vert \leq \vert \Omega_{\prec} \setminus S \vert $. Thus $\vert \mathrm{PF}(S) \vert \leq \vert \Omega_{\prec} \setminus S \vert + 1 $. Hence $\vert \mathrm{PF}(S) \vert = \vert \Omega_{\prec} \setminus S \vert + 1$. Conversely, suppose $\vert \mathrm{PF}(S) \vert = \vert \Omega_{\prec} \setminus S \vert + 1$. Since $\mathrm{PF}_{\substack{\\\prec}}' (S) \subseteq  \Omega_{\prec}$, we get $ \Omega_{\prec} \setminus S = \mathrm{PF}_{\substack{\\\prec}}' (S)$. This implies $ \Omega_{\prec} = S \cup \mathrm{PF}_{\substack{\\\prec}}' (S)$, and hence $\Omega_{\prec} + S^* = S^*$.

\end{proof}

\section{Extended Wilf's conjecture}
 For a numerical semigroup $S,$ Wilf in \cite{wilf}, proposed a conjecture related to the Diophantine Frobenius Problem that claims the inequality 
\[ \F(S)+1 \leq e(S) \cdot |\{ s \in S \mid s < \F(S) \}| \] 
is true. While this conjecture remains open, a potential extension of Wilf's conjecture to affine semigroups is studied in \cite{Wilfconjecture}.

\begin{definition}[{\cite[Definition 1]{Wilfconjecture}}]
 Let $S$ be a $\mathcal{C}$-semigroup and $\prec$ be a monomial order satisfying that every monomial is preceded only by a finite number of monomials. Define the Frobenius number of $S$ as $\mathcal{N}(\F(S)_{\substack{\\\prec}}) = | \mathcal{H}(S)| + | \{ g \in S \mid g \prec \F(S)_{\substack{\\\prec}}\}|.$
\end{definition}

Observe that if $S$ is a numerical semigroup, then $\mathcal{N}(\F(S)_{\substack{\\\prec}}) = \F(S),$ the Frobenius number of $S.$

\begin{conjecture} [{\bf Extension of Wilf's conjecture} {\cite[Conjecture 14]{Wilfconjecture}}]
	Let $S$ be a $\mathcal{C}$-semigroup and $\prec$ be a monomial order satisfying that every monomial is preceded only by a finite number of monomials. The extended Wilf's conjecture is
	\[ | \{ g \in S \mid g \prec \F(S)_{\substack{\\\prec}} \} | \cdot e(S) \geq \mathcal{N}(\F(S)_{\substack{\\\prec}}) + 1, \]
	where $e(S)$ denotes the embedding dimension of $S.$
\end{conjecture}

\begin{proposition}\label{wilfssufficient}
Let $S$ be a $\prec$-almost symmetric $\mathcal{C}$-semigroup with full cone. If \[ 2 ~ \vert \{ g \in S \mid g \prec \F(S)_{\substack{\\\prec}} \} \vert + \vert \mathrm{PF}(S) \vert \leq e(S) \cdot \vert \{ g \in S \mid g \prec \F(S)_{\substack{\\\prec}}\} \vert, \]
then $S$ satisfies the extended Wilf's conjecture.
\end{proposition}
\begin{proof}
Observe that
\begin{center}
$ \{ g \in S \mid g \preceq_{\mathbb{N}^d} \F(S)_{\substack{\\\prec}} \}  \subseteq \{ g \in S \mid g \prec \F(S)_{\substack{\\\prec}} \},$ for any term order $\prec.$
\end{center} 
	If $S$ is $\prec$-almost symmetric, then from Proposition \ref{cardH(S)}, we see that
	\begin{align*} 
	\mathcal{N}(\F(S)_{\substack{\\\prec}}) &
	= \vert \mathcal{H}(S) \vert + \vert \{ g \in S \mid g \prec \F(S)_{\substack{\\\prec}}\} \vert \\
	& = \vert \{ g \in S \mid g \preceq_{\mathbb{N}^d} \mathrm{F}(S)_{\substack{\\\prec}}\} \vert + \vert \mathrm{PF}_{\substack{\\\prec}}'(S) \vert + \vert \{ g \in S \mid g \prec \F(S)_{\substack{\\\prec}}\} \vert \\
	&\leq   2 ~ \vert \{ g \in S \mid g \prec \F(S)_{\substack{\\\prec}}\} \vert + \vert \mathrm{PF}_{\substack{\\\prec}}'(S) \vert.
	\end{align*}
Therefore, we get 
\begin{align*} 
	\mathcal{N}(\F(S)_{\substack{\\\prec}}) + 1 & \leq 2 ~ \vert \{ g \in S \mid g \prec \F(S)_{\substack{\\\prec}}\} \vert + \vert \mathrm{PF}(S) \vert.
	\end{align*}
\end{proof}

\begin{coro}\label{wilfsuffcor}
Let $S$ be a $\prec$-almost symmetric $\mathcal{C}$-semigroup with full cone. If 
\[\vert \mathrm{PF}(S) \vert \leq \vert \{ g \in S \mid g \prec \F(S)_{\substack{\\\prec}}\} \vert, \]
then $S$ satisfies the extended Wilf's conjecture.
\end{coro}
\begin{proof}
By Proposition \ref{wilfssufficient}, we have
\begin{align*} 
	\mathcal{N}(\F(S)_{\substack{\\\prec}}) + 1 & \leq 2 ~ \vert \{ g \in S \mid g \prec \F(S)_{\substack{\\\prec}}\} \vert + \vert \mathrm{PF}(S) \vert.
	\end{align*}
Since $\vert \mathrm{PF}(S) \vert \leq \vert \{ g \in S \mid g \prec \F(S)_{\substack{\\\prec}}\} \vert $, therefore we have
\begin{align*} 
	\mathcal{N}(\F(S)_{\substack{\\\prec}}) + 1 & \leq 3 ~ \vert \{ g \in S \mid g \prec \F(S)_{\substack{\\\prec}}\} \vert.
	\end{align*}
Since $S$ is a $\mathcal{C}$-semigroup with full cone, $e(S)$ has to be greater than or equal to $3$. This completes the proof.
\end{proof}

Note that if $S$ is an affine semigroup of $\mathbb{N}^d,$ where $d \geq 2,$ then the semigroup ring $k[S]$ is Cohen-Macaulay when $e(S)=2.$ Since our affine semigroups are MPD, we may assume that $e(S) \geq 3.$

\begin{theorem}\label{wilfbettithree}
Let $S$ be a $\prec$-almost symmetric $\mathcal{C}$-semigroup with full cone. If the Betti-type of $S$ is three, then $S$ satisfies the extended Wilf's conjecture.
\end{theorem}
\begin{proof}
If $\vert \{ g \in S \mid g \prec \F(S)_{\substack{\\\prec}}\} \vert \geq 3$, then by Corollary \ref{wilfsuffcor}, $S$ satisfies extended Wilf's conjecture. The conjecture is also satisfied when $\vert \{ g \in S \mid g \prec \F(S)_{\substack{\\\prec}}\} \vert = 1$ and $e(S) \geq 5$,  $\vert \{ g \in S \mid g \prec \F(S)_{\substack{\\\prec}}\} \vert = 2$ and $e(S) \geq 4$. We will now prove that the cases $\vert \{ g \in S \mid g \prec \F(S)_{\substack{\\\prec}}\} \vert = 1$ and $e(S) = 3 ~ \text{or} ~ 4$, $\vert \{ g \in S \mid g \prec \F(S)_{\substack{\\\prec}}\} \vert = 2$ and $e(S) = 3$ are not possible. This will complete the proof.

Suppose $\vert \{ g \in S \mid g \prec \F(S)_{\substack{\\\prec}}\} \vert = 1$. By Proposition \ref{cardH(S)}, we get $\mathcal{H}(S) = \mathrm{PF}(S)$. We also have, ${\bf{0}} = (0,\ldots,0) \prec \F(S)_{\substack{\\\prec}}$ for each term order $\prec$ and $\vert \mathcal{H}(S) \vert = \vert \mathrm{PF}(S) \vert = 3$. Since $S$ is $\prec$-almost symmetric, for $i \neq j \in [1,d]$ the only possible choices for $\mathcal{H}(S)$ are $\{e_i, e_j , e_i + e_j\}$ and $\{e_i, 2e_i, 3e_i\}$, where $\{e_1,\ldots, e_d\}$ is the standard $\mathbb{Q}$-basis of the vector space $\mathbb{Q}^d$. If $\mathcal{H}(S) = \{e_i, e_j , e_i + e_j\}$, then $S = \mathbb{N}^d \setminus \mathcal{H}(S)$ is minimally generated by the set
\[ \{ 2e_i, 3e_i, 2e_j, 3e_j\} \cup \{e_i + 2e_j, e_i+3e_j, 2e_i+e_j, 3e_i+e_j\}\bigcup \left( \bigcup_{\substack{k=1\\ k \neq i,j}}^{d} \{e_k\} \right). \]
 
Since $d \geq 2$, the embedding dimension, in this case, is at least $8$. Thus, in this case $e(S)$ can not be $3$ or $4$ whenever $\vert \{ g \in S \mid g \prec \F(S)_{\substack{\\\prec}}\} \vert = 1$. If $\mathcal{H}(S) = \{e_i, 2e_i, 3e_i\}$ then $S = \mathbb{N}^d \setminus \mathcal{H}(S)$ is minimally generated by the set
\[ \{ 4e_i, 5e_i, 6e_i, 7e_i\} \bigcup \left( \bigcup_{\substack{k=1\\ k \neq i}}^{d} \{e_j\} \right) \bigcup \left( \bigcup_{\substack{k=1\\ k \neq i}}^{d} \{ e_i+e_k, 2e_i + e_k, 3e_i + e_k \} \right). \]

Since $d \geq 2$, the embedding dimension in this case also is at least $8$. Thus in this case also, $e(S)$ can not be $3$ or $4$ whenever $\vert \{ g \in S \mid g \prec \F(S)_{\substack{\\\prec}}\} \vert = 1$.

Note that since $\mathcal{H}(S) \neq \emptyset$, this implies $e_i \in \mathcal{H}(S)$ for some $i \in [1,d]$. If $d \geq 3$, then $e(S)$ has to be greater or equal to $4$. If $d \geq 3$ then $e(S)$ can not be $3$. Therefore, we assume that $d = 2$. However, for $d = 2$, we can easily observe that if $\vert \mathcal{H}(S) \vert = 4$, then $e(S)$ has to be greater than $3$. This completes the proof. 
\end{proof}

\begin{example}
{\rm Let $S$ be a submonoid of $\mathbb{N}^2$ generated by the set 
$$\{(0,1),(4,0),(5,0),(6,0),(7,0),(1,4),(2,7),(3,10)\}.$$
 Therefore,
\[\mathcal{H}(S) = \left\lbrace 
	\begin{array}{c}
	(1,0),(1,1),(1,2),(1,3),(2,0),(2,1),(2,2),(2,3),(2,4),(2,5),(2,6) \\[1mm]
	(3,0),(3,1),(3,2),(3,3),(3,4),(3,5),(3,6),(3,7),(3,8),(3,9)
	\end{array}
	\right\rbrace .\]
Note that by \cite[Corollary 2.15]{op}, $\mathrm{PF}(S) = \{(1,3),(2,6),(3,9)\}$. Let $\prec$ be the graded lexicographic order then $S$ is $\prec$-almost symmetric with $\F(S)_{\substack{\\\prec}} = (3,9)$.

\begin{figure}
    \centering
    \resizebox{8cm}{8cm}{%
    \begin{tikzpicture}
     \draw [blue!10,thick,fill=blue!10] (0,10) rectangle (13,13);
     \draw [blue!10,thick,fill=blue!10] (3,0) rectangle (13,13);
    \foreach \x in {0,...,12}
     { \draw[dotted] (0,\x)--(13,\x);
      \node at (-1,\x){$\x$};
      }
    \foreach \x in {0,...,12}
     { \draw[dotted](\x,0)--(\x,13);
       \node at (\x,-1){$\x$};
      }
    \draw[thick,->] (0,0)--(0,13);
    \draw[thick,->] (0,0)--(13,0);
    \draw[fill=green] (0,0) circle (0.05cm);
    \draw[fill=green] (0,1) circle (0.1cm);
    \draw[fill=green] (4,0) circle (0.1cm);
    \draw[fill=green] (5,0) circle (0.1cm);
    \draw[fill=green] (6,0) circle (0.1cm);
    \draw[fill=green] (1,4) circle (0.1cm);
    \draw[fill=green] (2,7) circle (0.1cm);
    \draw[fill=green] (3,10) circle (0.1cm);
    \draw[fill=green] (7,0) circle (0.1cm);
    \draw[fill=green] (0,2) circle (0.05cm);
    \draw[fill=green] (0,3) circle (0.05cm); \draw[fill=green] (0,4) circle (0.05cm); \draw[fill=green] (0,5) circle (0.05cm); \draw[fill=green] (0,6) circle (0.05cm); \draw[fill=green] (0,7) circle (0.05cm); \draw[fill=green] (0,8) circle (0.05cm); \draw[fill=green] (0,9) circle (0.05cm); \draw[fill=green] (1,5) circle (0.05cm); \draw[fill=green] (1,6) circle (0.05cm); \draw[fill=green] (1,7) circle (0.05cm); \draw[fill=green] (1,8) circle (0.05cm); \draw[fill=green] (1,9) circle (0.05cm);
     \draw[fill=green] (2,8) circle (0.05cm);
      \draw[fill=green] (2,9) circle (0.05cm);
       \draw[fill=green] (0,10) circle (0.05cm); \draw[fill=green] (1,10) circle (0.05cm); \draw[fill=green] (2,10) circle (0.05cm);
       \draw[fill=red] (1,2) circle (0.05cm);  \draw[fill=red] (1,0) circle (0.05cm);  \draw[fill=red] (2,0) circle (0.05cm);  \draw[fill=red] (3,0) circle (0.05cm); \draw[fill=red] (1,1) circle (0.05cm); \draw[fill=red] (1,3) circle (0.05cm); \draw[fill=red] (2,1) circle (0.05cm); \draw[fill=red] (2,2) circle (0.05cm); \draw[fill=red] (2,3) circle (0.05cm); \draw[fill=red] (2,4) circle (0.05cm); \draw[fill=red] (2,5) circle (0.05cm); \draw[fill=red] (2,6) circle (0.05cm); \draw[fill=red] (3,1) circle (0.05cm); \draw[fill=red] (3,2) circle (0.05cm); \draw[fill=red] (3,3) circle (0.05cm);
       \draw[fill=red] (3,4) circle (0.05cm)
       ;\draw[fill=red] (3,5) circle (0.05cm);
       \draw[fill=red] (3,6) circle (0.05cm);
       \draw[fill=red] (3,7) circle (0.05cm);
       \draw[fill=red] (3,8) circle (0.05cm);
       \draw[fill=red] (3,9) circle (0.05cm);
    \end{tikzpicture}
    }
\caption{}
    \label{fig:my_label}
\end{figure}
The green points and the points under the shaded blue area are the elements of the semigroup $S$. The red points are the elements of $\mathcal{H}(S)$. The big green points are representing the generators. Note that \[
 \vert \mathcal{H}(S) \setminus  \mathrm{PF}_{\substack{\\\prec}}'(S)  \vert = \vert \{ g \in S \mid g \preceq_{\mathbb{N}^d} \mathrm{F}(S)_{\substack{\\\prec}}\} \vert = 19. 
\]
Since $\prec$ is graded lexicographic order, we have 
\[ 
\{ g \in S \mid g \prec \F(S)_{\substack{\\\prec}}\} = \left\lbrace \{(a,b) \mid a+b \leq 11\} \setminus \mathcal{H}(S) \right\rbrace \cup \{(0,12),(1,11),(2,10)\}.
\]
Observe that $\vert \{ g \in S \mid g \prec \F(S)_{\substack{\\\prec}}\} \vert = 61$, therefore

 \[ \mathcal{N}(\F(S)_{\substack{\\\prec}}) = \vert \{ g \in S \mid g \prec \F(S)_{\substack{\\\prec}}\} \vert + \vert \mathcal{H}(S) \vert = 82. \]
 Thus,
\[ \mathcal{N}(\F(S)_{\substack{\\\prec}}) +1 = 83 < 488 = e(S) \cdot \vert \{ g \in S \mid g \prec \F(S)_{\substack{\\\prec}}\} \vert . \]
Hence, $S$ satisfies extended Wilf's conjecture.}
\end{example}

\begin{definition}
If $S$ is $\mathcal{C}$-semigroup, we say that $S$ is irreducible if it cannot be expressed as an intersection of two finitely generated submonoids of $\mathbb{N}^d$. 
\end{definition}

\begin{coro}
Let $S$ be an irreducible $\mathcal{C}$-semigroup with full cone. Then $S$ satisfies the extended Wilf's conjecture.
\end{coro}
\begin{proof}
Follows from \cite[Theorem 21]{pfelements} and \cite[Theorem 3.12]{op}.
\end{proof}

\section{Unboundedness of Betti-type of MPD-semigroups}
In this section, we give a class of $\mathrm{MPD}$-semigroups of embedding dimension four such that there is no upper bound on the Betti-type of $\mathrm{MPD}$-semigroups in terms of embedding dimension. Thus, we conclude that the Betti-type of $\mathrm{MPD}$- semigroup may not be a bounded function of the embedding dimension.

\begin{definition}
For an element $0 \neq a \in S$, the Ap\'{e}ry set of $a$ is defined as 
\[  
\mathrm{Ap}(S,a)= \{b \in S \mid b-a \notin S\},
\]
 and for a subset $E$ of $S,$ the Ap\'{e}ry set of $E$ is defined as 
\[
\mathrm{Ap}(S,E)= \{b \in S \mid b-a \notin S, \forall a \in E\}.
\]
\end{definition}

\medskip
\noindent 
Let $a \geq 3$ be an odd natural number and $p \in \mathbb{Z}^+$. Now, similar to \cite[Example 3.7]{jafari}, we define
\[
S_{a,p} = \langle (a,0),(0,a^p),(a+2,2),(2,2+a^p) \rangle.
\]
$S_{a,p}$ is a simplicial affine semigroup in $\mathbb{N}^2$.

\vspace{1cm}
\medskip
\begin{lemma}\label{AperysetS_a,p}
	$\Ap(S_{a,p},\{(a,0),(0,a^p)\}) = \{\alpha (a+2,2) + \alpha' (2,2+a^p) \mid \alpha + \alpha' < a^p \}$.
\end{lemma}

\begin{proof}
	It is clear from the definition of $\mathrm{Ap}(S,E)$ that 
\[
\mathrm{Ap}(S_{a,p},\{(a,0),(0,a^p)\}) \subseteq \{\alpha (a+2,2) + \alpha' (2,2+a^p) \mid \alpha ,\alpha' \in \mathbb{N} \}.
\]
Suppose $\alpha + \alpha' \geq a^p,$ therefore we can write 
\begin{center}
$\alpha + \alpha' = \tilde{\alpha}+(\alpha - \tilde{\alpha}) + \tilde{\alpha'}+(\alpha' - \tilde{\alpha'})$ such that $\tilde{\alpha} + \tilde{\alpha'} = a^p$.
\end{center}
Therefore,
\begin{align*}
\alpha (a+2,2) + \alpha' (2,2+a^p) - (a,0) = \tilde{\alpha}(a+2,2) &+ \tilde{\alpha'} (2,2+a^p)
+ (\alpha - \tilde{\alpha})(a+2,2)\\
& +(\alpha' - \tilde{\alpha'})(2,2+a^p)-(a,0).
\end{align*}
This implies
\begin{align*}
\alpha (a+2,2) + \alpha' (2,2+a^p) - (a,0) = (2 a^{p-1}+&(\alpha-1))(a,0)+(2 + \alpha')(0,a^p)\\
&+(\alpha - \tilde{\alpha})(a+2,2)
+(\alpha' - \tilde{\alpha'})(2,2+a^p).
\end{align*}
Therefore, if  $\alpha + \alpha' \geq a^p$ then $\alpha (a+2,2) + \alpha' (2,2+a^p) \notin \mathrm{Ap}(S_{a,p},\{(a,0),(0,a^p)\})$.
Therefore, we get
\[
\mathrm{Ap}(S_{a,p},\{(a,0),(0,a^p)\}) \subseteq \{\alpha (a+2,2) + \alpha' (2,2+a^p) \mid \alpha + \alpha' < a^p \}.
\]
Now assume that there exist $\alpha$ and $\alpha'$ such that $\alpha + \alpha' < a^p$ and $\alpha (a+2,2) + \alpha' (2,2+a^p) \notin \mathrm{Ap}(S_{a,p},\{(a,0),(0,a^p)\})$. Therefore, either $\alpha (a+2,2) + \alpha' (2,2+a^p) - (a,0) \in S_{a,p}$ or $\alpha (a+2,2) + \alpha' (2,2+a^p) - (0,a^p) \in S_{a,p}$. Therefore, we can write
\[
\alpha (a+2,2) + \alpha' (2,2+a^p)= \lambda_1(a,0) + \lambda_2 (0,a^p) + \lambda_3 (a+2,2) + \lambda_4(2, 2+a^p),
\]
for $\lambda_1,\lambda_2,\lambda_3,\lambda_4 \in \mathbb{N}$ such that either of $\lambda_1$ or $\lambda_2$ is non-zero. Thus
\[
(\alpha - \lambda_3)(a+2,2) + (\alpha' - \lambda_4) (2,2+a^p) = \lambda_1(a,0) + \lambda_2 (0,a^p).
\]
This implies
\[
2(\alpha - \lambda_3 + \alpha' - \lambda_4) + (\alpha' - \lambda_4)a^p = \lambda_2 a^p 
\]
and
\[
(a+2)(\alpha - \lambda_3 ) + 2(\alpha'-\lambda_4) = \lambda_1 a
\]
Therefore, $2(\alpha+\alpha') = a^p(\lambda_2 + \lambda_4 - \alpha') + 2 (\lambda_3 + \lambda_4)$. Since, $a^p$ is odd, we have $\lambda_2 + \lambda_4 - \alpha'$ is even. Thus, there exist $\ell \in \mathbb{Z}$ such that $(\alpha+\alpha') = a^p \ell + \lambda_3 + \lambda_4$. Since $\alpha + \alpha' < a^p$, we get $\ell \leq 0$. Therefore, we get $\alpha - \lambda_3 + \alpha' - \lambda_4 \leq 0$. Since $\lambda_2 \geq 0$, we get $\alpha - \lambda_3 \leq 0$ and $\alpha' - \lambda_4 \geq 0$. Note that 
\[
a(\alpha - \lambda_3 ) + 2(\alpha - \lambda+3 + \alpha'-\lambda_4) = \lambda_1 a.
\]
Since $\alpha - \lambda_3 + \alpha' - \lambda_4 \leq 0$, $\alpha - \lambda_3 \leq 0$, and $\lambda_1 \geq 0$, we get 
\[
\lambda_1 = 0 = \alpha - \lambda_3 + \alpha' - \lambda_4 =  \alpha - \lambda_3.
\] 
This implies $\lambda_2$ is also zero, which is a contradiction. Hence
\[
 \{\alpha (a+2,2) + \alpha' (2,2+a^p) \mid \alpha + \alpha' < a^p \} \subseteq \mathrm{Ap}(S_{a,p},\{(a,0),(0,a^p)\}).
\]
This completes the proof.
\end{proof}

\medskip
Define the set
 $$\Delta = \left\lbrace \left( a^p(a+2)-(\ell + 2)a - 2, a^p(\ell +2)-2 \right) \mid 0 \leq \ell < a^p - 1  \right\rbrace .$$
 
\medskip 
\begin{proposition}\label{unboundedbettitype}
$S_{a,p}$ is an $\mathrm{MPD}$-semigroup and $ \Delta \subseteq \mathrm{PF}(S_{a,p})$. In particular, the Betti-type of $S_{a,p}$ is unbounded with respect to its embedding dimension. 
\end{proposition}
\begin{proof}
Set $f_{\ell} = \left( a^p(a+2)-(\ell + 2)a - 2, a^p(\ell +2)-2 \right)$ for $0 \leq \ell < a^p - 1 $. Therefore, we have
\[
f_{\ell}+(a,0) = (a^p-\ell-1 ) (a+2,2) + \ell (2,a^p+2),
\]
\medskip
\[
f_{\ell}+(0,a^p) = (a^p-\ell-2 ) (a+2,2) + (\ell+1) (2,a^p+2),
\]
\medskip
\[
f_{\ell}+(a+2,2) = \left(a^{p-1}(a+2)-\ell-1\right)(a,0) + (\ell + 2)(0,a^p), 
\]
\medskip
\[
f_{\ell}+(2,2+a^p) = \left(a^{p-1}(a+2)-\ell\right)(a,0) + (\ell + 3)(0,a^p). 
\]
\medskip

This implies that for any $f \in \Delta$, $f + s \in S_{a,p}$ for all $s \in S_{a,p} \setminus \{0\}$. To complete the proof, it remains to prove that $\Delta \subseteq \mathcal{H}(S_{a,p})$. Let $f \in \Delta$,
 then $$f = \left( a^p(a+2)-(\ell + 2)a - 2, a^p(\ell +2)-2 \right) ~ \text{for some} ~ \ell \in [0,a^p-2].$$ Thus 
\[
f + (0,a^p) = (a^p-\ell-2 ) (a+2,2) + (\ell + 1) (2,a^p+2) ~ \text{for some} ~ \ell \in [0,a^p-2].
 \]
Since $(a^p-\ell-2 ) + (\ell + 1) < a^p$, therefore by Lemma \ref{AperysetS_a,p}, we have $$f + (0,a^p) \in \mathrm{Ap}(S_{a,p}, \{(a,0),(0,a^p)\}).$$ Therefore $f  \notin S_{a,p}$. It is clear that $\Delta \subseteq \mathbb{N}^2 = \mathrm{cone}(S_{a,p}) \cap \mathbb{N}^2$, and hence $\Delta \subseteq \mathrm{PF}(S_{a,p})$.  Now, as the Betti-type of an $\mathrm{MPD}$-semigroup $S$ is equal to the cardinality of the set $\mathrm{PF}(S)$, we get $\vert \mathrm{PF}(S_{a,p}) \vert \geq a^p - 1$.
\end{proof}

Now, we recall the definition of gluing of two affine semigroups.

\begin{definition}
Let $S \subseteq \mathbb{N}^d$ be an affine semigroup and $G(S)$ be the group spanned by $S$, that is, $G(S) =\{ a-b \in \mathbb{Z}^d \mid a, b \in S\}$. Let $A$ be the minimal generating system of $S$ and $A = A_1 \amalg A_2$ be a nontrivial partition of $A$. Let $S_i$ be the submonoid of $\mathbb{N}^r$ generated by $A_i, i \in {1, 2}$. Then $S = S_1 + S_2.$ We say that $S$ is the gluing of $S_1$ and $S_2$ by $s$ if 
	\begin{enumerate}
		\item $s \in S_1 \cap S_2$ and,
		\item $G(S_1 ) \cap G(S_2) = s\mathbb{Z}$. 
	\end{enumerate}
	If $S$ is a gluing of $S_1$ and $S_2$ by $d,$ we write $S = S_1 +_s S_2$.
\end{definition} 
 
\begin{theorem}[{\cite[Theorem 17]{pfelements}, \cite[Theorem 2.6]{op}}] \label{PFGluing}
	Let $S$ be an affine semigroup such that $S = S_1 +_s S_2,$ where $S_1$ and $S_2$ are $\mathrm{MPD}$-semigroups. Then $S$ is also an $\mathrm{MPD}$-semigroup. Moreover,
	\begin{align*}
		\PF(S)=\{f+g+s \mid f \in \PF(S_1), \ g\in \PF(S_2) \}.
	\end{align*} 
\end{theorem} 
 
The converse of the Theorem \ref{PFGluing} is also true.

\begin{proposition}
Let $S$ be a gluing of $S_1 = \langle a_1, \ldots, a_{n_1} \rangle$ and $S_2 = \langle a_1', \ldots, a_{n_2}' \rangle$. Then $S$ is an $\mathrm{MPD}$-semigroup if and only if $S_1$ and $S_2$ are $\mathrm{MPD}$-semigroups.
\end{proposition}
\begin{proof}
If $S_1$ and $S_2$ are $\mathrm{MPD}$-semigroups then by \cite[Theorem 17]{pfelements}, $S$ is an $\mathrm{MPD}$-semigroup. For the converse part, suppose without loss of generality that $S_1$ is not an $\mathrm{MPD}$-semigroup. Therefore, we have  $\mathrm{pd}_{R_1}k[S_1] < n_1-1$, where $R_1 = k[x_1, \ldots,x_{n_1}]$. Also, by Auslander-Buchsbaum formula, we have $\mathrm{pd}_{R_2} k[S_2] \leq n_2-1$, where $R_2 = k[x_1, \ldots, x_{n_2}]$. Therefore by \cite[Corollary 6.2]{hema2019}, for $R = k[x_1, \ldots, x_{n_1+n_2}]$, we have 
\begin{align*}
\mathrm{pd}_R k[S] = \mathrm{pd}_{R_1} k[S_1] + \mathrm{pd}_{R_2} k[S_2] + 1 < n_1-1+n_2-1+1 =n_1+n_2-1.
\end{align*}
Thus, we have $\mathrm{depth}_R k[S] \geq 2$. Hence, $S$ is not an $\mathrm{MPD}$-semigroup. This is a contradiction.
\end{proof}

Now, using the technique of gluing, we give a class of $\mathrm{MPD}$-semigroups in $\mathbb{N}^2$ of each embedding dimension $e \geq 4$, where there is no upper bound on the Betti-type in terms of the embedding dimension. 

\begin{remark}\label{remarkunbdd}
For each $e \geq 4$, there exists a class of $\mathrm{MPD}$-semigroups of embedding dimension $e$ in $\mathbb{N}^2$, where there is no upper bound on the Betti-type in terms of the embedding dimension $e$.
\end{remark}
\begin{proof}
Let $a \geq 3$ be an odd natural number and $p \in \mathbb{Z}^+$. Define
\[
S_{a,p} = \langle (a,0),(0,a^p),(a+2,2),(2,2+a^p) \rangle.
\]
Now, let $S_1 = \langle n_1, \ldots, n_s \rangle$ be a numerical semigroup. With $a$ and $p$ as above, define 
\[
\tilde{S_1} = \langle n_1(a, a^p), n_2(a,a^p), \ldots, n_s(a,a^p)\rangle.
\]
Let $\mu = \sum_{i=1}^s n_i$, and define 
\[
\tilde{S_{a,p}} = \langle \mu(a,0),\mu(0,a^p),\mu(a+2,2),\mu(2,2+a^p) \rangle.
\]
Now define,
\[
S_{a,p,s} = \langle \mu(a,0),\mu(0,a^p),\mu(a+2,2),\mu(2,2+a^p), n_1(a, a^p), n_2(a,a^p), \ldots, n_s(a,a^p) \rangle.
\]
Observe that the embedding dimension of $S_{a,p,s}$ is $s+4$. By \cite[Proposition 2]{hemagluingwhenandhow}, note that $S_{a,p,s}$ is a gluing of $\tilde{S_{a,p}}$ and $\tilde{S_1}$. Let $\PF(S_1)$ be the set of pseudo-Frobenius elements of the numerical semigroup $S_1$. It is easy to observe that $(a,a^p)\PF(S_1)$ is the set of pseudo-Frobenius elements of $\tilde{S_1}$. Let $\nu$ be the cardinality of $\PF(S_1)$. Therefore by Theorem \ref{PFGluing} and Proposition \ref{unboundedbettitype}, the cardinality of the set of pseudo-Frobenius elements of $S_{a,p,s}$ is greater or equal to $\nu(a^p-1)$. Thus the Betti-type of $S_{a,p,s}$ depends on $a$ and $p$. Hence, the Betti-type has no upper bound in terms of embedding dimension.
\end{proof}

Note that the classes of semigroups we have given above are not $\mathcal{C}$-semigroups. We ask the following question about Betti-type of $\mathcal{C}$-semigroups.

\begin{question}
Let $d >1$, what can one say about the boundedness of the Betti-type of $\mathcal{C}$-semigroups of $\mathbb{N}^d $ in terms of the embedding dimension $?$
\end{question}

\section{Submonoids with Arf property}

In \cite{lipmanArf}, Lipman gave a characterization of Arf rings in terms of their value semigroup. This motivated the study of Arf numerical semigroups in the semigroup theoretic sense(e.g., see \cite{rosalesArf}). A numerical semigroup $S$ is called Arf if it satisfies the following property:
\[
\text{for}~ x,y,z \in S, ~ \text{if} ~ x \geq y \geq z ~ \text{then} ~ x+y-z \in S.
\]
This property is called the Arf property. In this section, we study those submonoids of $\mathbb{N}^d$ which satisfy the Arf property with respect to the usual partial order on $\mathbb{N}^d$.

\begin{definition}
 A submonoid $S$ of $\mathbb{N}^d$ is called Arf if for all $x \preceq_{\mathbb{N}^d} y \preceq_{\mathbb{N}^d} z$, $y + z - x \in S$.
\end{definition}

\begin{example}
{ \rm Let $S = \langle (0,1),(1,2),(2,0),(3,0) \rangle$. Then $S$ is Arf monoid. Let $x,y,z \in S$ such that $x \preceq_{\mathbb{N}^d} y \preceq_{\mathbb{N}^d} z$. If $y-x \in S$, we are done. If not then $y-x$ has to be either $(1,0)$ or $(1,1)$. If $y-x = (1,1)$ then for any $z \in S\setminus \{0\}$, we have $y+z-x \in S$. If $y - x = (1,0)$, then $y = (1,0)+x \preceq_{\mathbb{N}^d} z$. Therefore, $ (1+x_1, x_2) \preceq_{\mathbb{N}^d} z$, implies $(2+x_1,x_2) \preceq_{\mathbb{N}^d} y+z-x$. Hence, $y+z-x \in S$ for all $x,y,z \in S$.}
\end{example}

\begin{proposition}\label{a+SU0arf}
Let $S$ be a submonoid of $\mathbb{N}^d$ and $a \in S^*$. Then $S$ is an Arf monoid if and only if $(a+S) \cup \{0\}$ is an Arf monoid.
\end{proposition}
\begin{proof}
Suppose $S$ is an Arf monoid. Let $x,y,z \in a+S$ be such that $x \preceq_{\mathbb{N}^d} y \preceq_{\mathbb{N}^d} z$. Since $x,y,z \in a+S$, we have $x = a+s_1$, $y = a+s_2$, $z = a+s_3$ for some $s_1, s_2, s_3 \in S$. Therefore, $a+s_1 \preceq_{\mathbb{N}^d} a+s_2 \preceq_{\mathbb{N}^d} a+s_3$ implies $s_1 \preceq_{\mathbb{N}^d} s_2 \preceq_{\mathbb{N}^d} s_3$. Since $S$ is an Arf monoid, we have $s_2+s_3-s_1 \in S$. Therefore, $y + z -x = a+ s_2 + s_3 - s_1 \in a+S$ and hence, $(a+S) \cup \{0\}$ is an Arf monoid.  Conversely, suppose $(a+S) \cup \{0\}$ is an Arf monoid and $s_1, s_2, s_3 \in S$ be such that $s_1 \preceq_{\mathbb{N}^d} s_2 \preceq_{\mathbb{N}^d} s_3$. This implies $a+s_1 \preceq_{\mathbb{N}^d} a+s_2 \preceq_{\mathbb{N}^d} a+s_3$. Therefore, we have $a+ s_2 + s_3 - s_1 \in a+S$. Therefore, $s_2 + s_3 - s_1 \in S$ and hence, $S$ is an Arf monoid.
\end{proof}

It is clear from the definition that the intersection of  Arf monoids is again an arf monoid. We call the intersection of all Arf submonoids containing a submonoid $S$, the Arf closure of $S$, denoted by $\mathrm{Arf}(S)$. If a submonoid is contained in only finitely many Arf monoids, then the Arf closure of $S$ is the smallest (with respect to inclusion) Arf monoid containing $S$.

\begin{proposition}\label{arfclosure}
Let $S$ be a submonoid of $\mathbb{N}^d$. Then
\[
S' = \{y+z-x \mid x,y,z \in S, x \preceq_{\mathbb{N}^d} y \preceq_{\mathbb{N}^d} z\}
\]
is a submonoid of $\mathbb{N}^d$ and $S \subseteq S'$.
\end{proposition}
\begin{proof}
It is clear that $S' \subseteq \mathbb{N}^d$ and $S \subseteq S'$. Let $s_1, s_2 \in S'$. Then there exist $x_1,y_1,z_1,x_2,y_2,z_2 \in S$ such that $x_1 \preceq_{\mathbb{N}^d} y_1 \preceq_{\mathbb{N}^d} z_1$ and $x_2 \preceq_{\mathbb{N}^d} y_2 \preceq_{\mathbb{N}^d} z_2$, and $s_1 = y_1+z_1-x_1$, $s_2 = y_2+z_2-x_2$. Note that $x_1+x_2$, $y_1+y_2$, $z_1+z_2$ belong to $S$ and $x_1+x_2 \preceq_{\mathbb{N}^d} y_1+y_2 \preceq_{\mathbb{N}^d} z_1+z_2$. Therefore, $s_1+s_2 = (y_1+y_2) + (z_1+z_2) - (x_1+x_2) \in S'$. Hence $S'$ is a submonoid of $\mathbb{N}^d$.
\end{proof}

\begin{coro}\label{arfS=S'}
Let $S$ be a submonoid of $\mathbb{N}^d$. Then $S$ is Arf monoid if and only if $S = S'$.
\end{coro}

Now, recursively define $S^0 = S$ and $S^{k+1} = (S^k)'$ for any $k \in \mathbb{N}$.

\begin{proposition}\label{Si=ArfS}
Let $S$ be a submonoids of $\mathbb{N}^d$ contained in only finitely many submonoids of $\mathbb{N}^d$. Then there exist $i \in \mathbb{N}$ such that $S^i = \mathrm{Arf}(S)$.
\end{proposition}
\begin{proof}
By Proposition \ref{arfclosure}, it is clear that $S \subseteq S^k$ and $S^k \subseteq S^{k+1}$ for all $k \in \mathbb{N}$. Since, $S$ is contained in only finitely many submonoids of $\mathbb{N}^d$, we have $S^i = S^{i+1} = (S^i)'$ for some $i \in \mathbb{N}$. By Corollary \ref{arfS=S'}, $S^i$ is an Arf submonoid of $\mathbb{N}^d$ containing $S$. Since $\mathrm{Arf}(S)$ is the smallest Arf monoid containing $S$, we have $\mathrm{Arf}(S) \subseteq S^i$. On the other hand, it is obvious that $S^k \subseteq \mathrm{Arf}(S)$ for all $k \in \mathbb{N}$. Hence, we have $\mathrm{Arf}(S) = S^i$. 
\end{proof}

\begin{coro}
Let $S$ be a $\mathcal{C}$-semigroup with full cone. Then there exist $i \in \mathbb{N}$ such that $S^i = \mathrm{Arf}(S)$.
\end{coro}

\begin{example}
{\rm Let $S$ be the submonoid of $\mathbb{N}^d$ generated by the set $$\{(0,1),(3,0),(5,0),(1,3),(2,3)\}.$$
Therefore,
\[S = \mathbb{N}^2 \setminus \left\lbrace 
	\begin{array}{c}
	(1,0),(1,1),(1,2),(2,0),(2,1),(2,2) \\[1mm]
	(4,0),(4,1),(4,2),(7,0),(7,1),(7,2)
	\end{array}
	\right\rbrace .\] 
Observe that, 
\[S' = \mathbb{N}^2 \setminus \left\lbrace 
	\begin{array}{c}
	(1,0),(1,1),(1,2),(2,0),(2,1),(2,2) \\[1mm]
	(4,0),(4,1),(4,2)
	\end{array}
	\right\rbrace .\]
Since $S \subsetneq S'$, this implies $S$ is not an Arf monoid. Also, note that $S'$ is an Arf monoid, hence $\mathrm{Arf}(S) = S'$.}
\end{example}

\begin{proposition}\label{a+arf}
Let $S = \langle a,a_1,\ldots,a_n \rangle \subseteq \mathbb{N}^d$ be a submonoid such that $\{a_1,\ldots,a_n\}$ is a part of a chain in the poset $(\mathbb{N}^d, \preceq_{\mathbb{N}^d} )$. If $a_l \preceq_{\mathbb{N}^d} a_i+a_j$ for all $i,j,l \in [1,n]$, then 
\[
a+ \langle a,a_1,\ldots,a_n \rangle^k \subseteq \mathrm{Arf}(\langle a,a+a_1,\ldots,a+a_n  \rangle).
\]
\end{proposition}
\begin{proof}
Without loss of generality, assume that $a_1 \preceq_{\mathbb{N}^d} a_2 \preceq_{\mathbb{N}^d} \ldots \preceq_{\mathbb{N}^d} a_n$. For $k =0$, we need to prove that $a+ \langle a,a_1,\ldots,a_n \rangle \subseteq \mathrm{Arf}(\langle a,a+a_1,\ldots,a+a_n  \rangle)$. Let $i,j \in [1,n]$, then it is clear that $a+a_i, a+a_j \in \mathrm{Arf}(\langle a,a+a_1,\ldots,a+a_n  \rangle)$. Since, either $i \leq j$ or $j \leq i$, we have either $a_i \preceq_{\mathbb{N}^d} a_j$ or $a_j \preceq_{\mathbb{N}^d} a_i$. Without loss of generality assume that $i \leq j$, therefore we have $a \preceq_{\mathbb{N}^d} a+a_i \preceq_{\mathbb{N}^d} a+a_j$. This implies that $a + a_i + a_j = (a+a_i) + (a+a_j) - a \in \mathrm{Arf}(\langle a,a+a_1,\ldots,a+a_n  \rangle)$. Let $l \in [1,n]$, since $a_l \preceq_{\mathbb{N}^d} a_i+a_j$, we have $ a \preceq_{\mathbb{N}^d} a+a_l \preceq_{\mathbb{N}^d} a+a_i+a_j$. As we have, $a, a+ a_l, a + a_i + a_j \in \mathrm{Arf}(\langle a,a+a_1,\ldots,a+a_n  \rangle)$, this implies $a+a_i+a_j+a_l = (a+a_l) + (a+ a_i + a_j) - a \in \mathrm{Arf}(\langle a,a+a_1,\ldots,a+a_n  \rangle)$. By repeating this arguement we get, $(\alpha+1)a + \alpha_1 a_1 + \ldots + \alpha_n a_n \in \mathrm{Arf}(\langle a,a+a_1,\ldots,a+a_n  \rangle)$ for all $\alpha, \alpha_1, \ldots, \alpha_n \in \mathbb{N}$. Hence, $a+ \langle a,a_1,\ldots,a_n \rangle \subseteq \mathrm{Arf}(\langle a,a+a_1,\ldots,a+a_n  \rangle)$.
Now, suppose that $a+ \langle a,a_1,\ldots,a_n \rangle^{k-1} \subseteq \mathrm{Arf}(\langle a,a+a_1,\ldots,a+a_n  \rangle)$. Let $s \in a+ \langle a,a_1,\ldots,a_n \rangle^k$. Then there exist $s' \in \langle a,a_1,\ldots,a_n \rangle^k$ such that $s = a+s'$. Therefore there exist $x, y , z \in \langle a,a_1,\ldots,a_n \rangle^{k-1}$ such that $x \preceq_{\mathbb{N}^d} y \preceq_{\mathbb{N}^d} z$ and $y + z - x = s'$.   Since $a+x, a+y , a+z \in a + \langle a,a_1,\ldots,a_n \rangle^{k-1}$, by induction hypothesis $ s = a + s' = a + y + z - x = (a+y) + (a+z) - (a + x) \in \mathrm{Arf}(\langle a,a+a_1,\ldots,a+a_n  \rangle)$.
\end{proof}

\begin{proposition}\label{7.10}
Let $S = \langle a,a_1,\ldots,a_n \rangle \subseteq \mathbb{N}^d$ be a submonoid contained in only finitely many submonoids of $\mathbb{N}^d$ such that $\{a_1,\ldots,a_n\}$ is a part of a chain in the poset $(\mathbb{N}^d, \preceq_{\mathbb{N}^d} )$. If $a_l \preceq_{\mathbb{N}^d} a_i+a_j$ for all $i,j,l \in [1,n]$, then 
\[
 \mathrm{Arf}(\langle a,a+a_1,\ldots,a+a_n  \rangle) = (a + \mathrm{Arf}(\langle a,a_1,\ldots,a_n  \rangle)) \cup \{0\}.
\]
\end{proposition}
\begin{proof}
By Proposition \ref{Si=ArfS}, there exist $i \in \mathbb{N}$ such that $\langle a,a_1,\ldots,a_n \rangle^i = \mathrm{Arf}(\langle a,a_1,\ldots,a_n \rangle)$. By Proposition \ref{a+arf}, $a + \mathrm{Arf}(\langle a,a_1,\ldots,a_n \rangle) \subseteq  \mathrm{Arf}(\langle a,a+a_1,\ldots,a+a_n  \rangle) $. Observe that $a, a+a_1, \ldots, a+a_n \in a + \mathrm{Arf}(\langle a,a_1,\ldots,a_n  \rangle)$. By proposition \ref{a+SU0arf}, we have $(a + \mathrm{Arf}(\langle a,a_1,\ldots,a_n  \rangle)) \cup \{0\}$ is an Arf monoid. Hence,  $\mathrm{Arf}(\langle a,a+a_1,\ldots,a+a_n  \rangle) \subseteq (a + \mathrm{Arf}(\langle a,a_1,\ldots,a_n  \rangle)) \cup \{0\}$.
\end{proof}

\begin{definition}
If S is a submonoid of  ${\mathbb{N}^d}$, we define the multiplicity of $S$ as $m(S)=
inf_{\preceq_{\mathbb{N}^d}}(S\setminus \{0\})$. A submonoid $S$ of $\mathbb{N}^d$ is said to be a $\mathrm{PI}$-monoid if there exist a submonoid $T$
of $\mathbb{N}^d$ and $a \in T \setminus\{0\}$ such that $S = a + T \cup \{0\}$.
\end{definition}

\begin{remark}
Let $S$ be a numerical semigroup minimally generated by $\{ n_1, \ldots ,n_k \}$ such that $k = n_1$. For $a \in \mathbb{N}^d$, Define $S' = \langle n_1a, \ldots , n_k a\rangle$. Then $S'$ is a $\mathrm{PI}$-monoid in $\mathbb{N}^d$. Because since $k = n_1$, $S$ is a maximal embedding dimension numerical semigroup. Therefore by \cite[Corollary 3.13]{numerical}, there exist a numerical semigroup $T$ such that $n_1 \in T$ and $S = (n_1 + T) \cup \{0\}$. Suppose $T$ is minimally generated by $\{n_1', \ldots, n_{k'}'\}$. Thus, we have $S' = n_1a + \langle n_1'a, \ldots, n_{k'}'a \rangle \cup \{0\}$. Hence, $S'$ is a $\mathrm{PI}$-monoid in $\mathbb{N}^d$.
\end{remark}

\begin{example}
{\rm Let $S$ be a submonoid of $\mathbb{N}^2$, generated by the set
\[
\{(6,12),(8,16),(9,18), (10,20), (11,22), (13,26)\}.
\]
Observe that $(6,12) = inf_{\preceq_{\mathbb{N}^d}}(S\setminus \{0\}) = m(S)$. Let $T$ be the submonoid of $\mathbb{N}^2$ generated by $\{(2,4),(3,6)\}$. Then it is clear that $S \subseteq ((6,12) + T) \cup \{0\}$. To see $((6,12) + T) \cup \{0\}  \subseteq S$, it is sufficient to see that $(6,12) + \alpha (2,4) + \beta (3,6) \in S$ for $0 \leq \alpha \leq 2$ and $0 \leq \beta \leq 1$. This is obvious, thus $S = ((6,12) + T) \cup \{0\}$. Hence $S$ is a $\mathrm{PI}$-monoid.}
\end{example}

\begin{proposition}
Let $S \subseteq \mathbb{N}^d$ be a submonoid such that $m(S) \in S^*$. Then $S$ is $\mathrm{PI}$-monoid if and only if for all $x,y \in S^*$, $x + y - m(S) \in S^*$.
\end{proposition}
\begin{proof}
Let $S$ be a PI-monoid. Suppose one of the $x$ or $y$ does not belong to the Ap\'{e}ry set of $m(S)$, then either $x-m(S) \in S$ or $y-m(S) \in S$, and hence $x+y-m(S) \in S^*$. Suppose $x$ and $y$ belong to the Ap\'{e}ry set of $m(S)$. By \cite[Corollary 39]{pfelements}, we can write $x = m(S) + f$ and $y = m(S)+f'$ for some $f,f' \in \mathrm{PF}(S)$. Therefore, $x+y-m(S) = m(S) + f + m(S) + f' - m(S) = f+f'+m(S) \in S^*$. Conversely, suppose, $x + y - m(S) \in S^*$, for all $x,y \in S^*$. Now, let $z \in S^*$, we can write $z = m(S)+z-m(S)$, therefore $S = m(S)+ (-m(S)+S^*) \cup \{0\}$. Since, for $-m(S)+x$ and $-m(S)+y$ from $-m(S)+S^*$, we have $-m(S)+x-m(S)+y = -m(S)+s$ for some $s \in S^*$. Therefore, $(-m(S)+S^*)$ is a submonoid of $\mathbb{N}^d$, and hence $S$ is a $\mathrm{PI}$-monoid.
\end{proof}

\begin{coro}\label{arf=pi}
Let $S$ be an Arf monoid such that $m(S) \in S^*$. Then $S$ is a $\mathrm{PI}$-monoid.
\end{coro}

\begin{example}
{\rm 

\begin{figure}
    \centering
     \resizebox{8cm}{8cm}{%
    \begin{tikzpicture}
    \draw [blue!10,thick,fill=blue!10] (1.9,1.9) rectangle (7,7);
 \foreach \x in {0,...,6}
     { \draw[dotted] (0,\x)--(7,\x);
      \node at (-1,\x){$\x$};
      }
    \foreach \x in {0,...,6}
     { \draw[dotted](\x,0)--(\x,7);
       \node at (\x,-1){$\x$};
      }
    \draw[thick,->] (0,0)--(0,7);
     \draw[thick,->] (0,0)--(7,0);
    \draw[fill=green] (0,0) circle (0.1cm);
    \draw[fill=red] (3,3) circle (0.05cm); 
    \draw[fill=red] (3,2) circle (0.05cm); 
\end{tikzpicture}
}
    \caption{}
    \label{fig:my_label2}
\end{figure}
Let $S$ be the submonoid of $\mathbb{N}^2$ containing all the lattice points in the shaded blue region except the red points (see Figure \ref{fig:my_label2}). Clearly, $S$ is not $\mathcal{C}$-semigroup. In fact, it is not a finitely generated submonoid. Observe that $inf_{{\preceq}_{\mathbb{N}^2}} (S\setminus \{0\}) = (2,2)= m(S)$. Hence $m(S) \in S^*$. Let $x,y,z \in S$ such that $x \preceq_{\mathbb{N}^d} y \preceq_{\mathbb{N}^d} z$. If $y-x \in S$, we are done. If not then $y-x$ has to be either $(3,2)$ or $(3,3)$. If $y-x = (3,3)$ then for any $z \in S\setminus \{0\}$, we have $y+z-x \in S$. If $y - x = (3,2)$, then $y = (3,2)+x \preceq_{\mathbb{N}^d} z$. Therefore, $ (3+x_1, 2+x_2) \preceq_{\mathbb{N}^d} z$, implies $(3+x_1,2+x_2) \preceq_{\mathbb{N}^d} y+z-x$. Hence, $y+z-x \in S$ for all $x,y,z \in S$. Thus, $S$ is an Arf monoid. Hence, by Corollary \ref{arf=pi}, $S$ is a $\mathrm{PI}$-monoid.}
\end{example}

\bibliographystyle{plain}

\end{document}